\documentclass[12pt]{article}
\usepackage{epsfig, color}
\usepackage{graphicx}
\usepackage[cp437de]{inputenc}
\usepackage{amsfonts,amsmath,amsthm}
\usepackage{times}
\usepackage{bm}
\usepackage{natbib}

\usepackage[plain,noend]{algorithm2e}
\newcommand{\R}{\mathbb{R}}
\newcommand{\Z}{\mathbb{Z}}
\newcommand{\N}{\mathbb{N}}

\newcommand{\var}{\mbox{Var}}
\newcommand{\Cov}{\mbox{Cov}}
\newcommand{\med}{\mbox{Med}}

\newcommand{\claw}{\longrightarrow}
\newcommand{\A}{{\mathcal A}}
\newcommand{\B}{{\mathcal B}}
\newcommand{\F}{{\mathcal F}}
\newcommand{\G}{{\mathcal G}}
\newcommand{\Hy}{{\mathcal H}}

\newcommand{\pr}{\mbox{pr}}

\newcommand{\cddot}{.}
\newtheorem{theorem}{Theorem}
\newtheorem{corollary}{Corollary}
\newtheorem{definition}{Definition}
\newtheorem{remark}{Remark}
\newtheorem{assumption}{Assumption}
\newtheorem{lemma}{Lemma}
\newtheorem{proposition}{Proposition}

\begin{document}
\begin{center}\Large \bf A robust method for shift detection in time series\end{center}

\begin{center}
H. Dehling$^a$, R. Fried$^b$, M. Wendler$^c$\\
${}^a$ Fakult\"at f\"ur Mathematik, Ruhr-Universit\"at Bo\-chum,
44780 Bochum, Germany, herold.dehling@rub.de\\
${}^b$ Fakult\"at Statistik, TU Dortmund University,
 44221 Dortmund, Germany, {fried@statistik.tu-dortmund.de}\\
${}^c$ Institut f\"ur Mathematik und Informatik, Universit\"at Greifswald\\
17487 Greifswald, Germany, {martin.wendler@uni-greifswald.de}
\end{center}

\begin{quote}
Abstract: We present a robust test for a change-point in time series which is based on the two-sample Hodges--Lehmann estimator.  We develop new limit theory for a class of statistics based on two-sample U-quantile processes, in the case of short range dependent observations. Using this theory we derive the asymptotic distribution of our test statistic under the null hypothesis of a constant level. The proposed test shows better overall performance under normal, heavy-tailed and skewed distributions than several other modifications of the popular cumulative sums test based on U-statistics, one-sample U-quantiles or M-estimation. The new theory does not involve moment conditions, so that any transform of the observed process can be used to test the stability of higher order characteristics such as variability, skewness or curtosis.\\[5mm]
Key words: Change-point tests; Functional central limit theorem; Hodges--Lehmann estimator; Two-sample U-process; Two-sample U-quantiles; Two-sample U-statistics; Weak dependence.
\end{quote}

\newpage

\section{Introduction}

Statistical tests for the presence of a change in the structure of a time series are of great importance, for instance regarding economic, technological and  climate data. Many procedures for detecting changes and for estimating change-points have been proposed in the literature, see \citet{csorgoe1997}. There has been a growing interest in change-point analysis for heavy-tailed time series recently, see \citet{huskova2012}, \citet{praskova2014}, \citet{chakar2017}, \citet{vogel2017} and \citet{fearnhead2018}.

    We study tests for detecting a level shift in a time series $(X_i: i\in\Z)$, assuming that
\[
 X_i=\mu_i + Y_i,
\]
where $(\mu_i:i\in \Z)$ is a sequence of unknown constants and $(Y_i:i\in\Z)$ is a stationary process with mean zero.  We will focus on the case when $(Y_i:i\in\Z)$ is a weakly dependent process, in a sense that we will specify below.
Given observations $X_1,\ldots, X_n$, we want to test the  null hypothesis that the process is stationary, that is
\[
 \Hy_0:\,\mu_1=\ldots=\mu_n,
\]
against the alternative that there is a level shift after some unknown point in time $k^\star$, that is
\[
 \Hy_1:\,   \mu_1=\ldots=\mu_{k^\star} \neq \mu_{k^\star+1}=\ldots=\mu_n,\quad k^\star\in \{1,\ldots,n-1\}.
\]
If the change-point $k^\star$ was known in advance, this would be a standard two-sample problem with samples $X_1,\ldots,X_{k^\star}$ and $X_{k^\star+1},\ldots,X_n$. Tests for the two-sample problem serve as guideline for finding tests for the
more difficult change-point problem studied here where $k^\star$ is unknown.

The standard statistic for the change-point problem is the cumulative sum statistic $C_n$, which can be written as the maximum of
\[
C_{n,k}= \frac{1}{{n}^{1/2}} \left( \sum_{i=1}^k X_i -\frac{k}{n} \sum_{i=1}^n X_i \right)
 ={n}^{1/2} \frac{k}{n} \left(1-\frac{k}{n}\right) \left(\frac{1}{k}\sum_{i=1}^k X_i - \frac{1}{n-k} \sum_{i=k+1}^n
 X_i  \right)
\]
over all candidate split time points $k$. The last term 
on the right hand side is the difference of the sample means, or equivalently the mean pairwise difference between the samples, which is the
standard estimator for a location shift if the two samples $X_1,\ldots,X_k$ and $X_{k+1},\ldots,X_n$ are Gaussian.
 The cumulative sums test uses the test statistic $C_n/\hat{\sigma}_n$, where $\hat{\sigma}_n^2$ is a consistent estimator of the long run variance $\sigma^2=\sum_{j=-\infty}^\infty \Cov(Y_0,Y_j)$ of the mean.
   The asymptotic distribution of this test statistic under $\Hy_0$ can be derived from a functional central limit theorem for the partial sum process $(n^{-1/2} \sum_{i=1}^{\lfloor n\lambda\rfloor} Y_i:\ {0\leq \lambda \leq 1})$.
     Approximate critical values can hence be obtained from tables of the
Kolmogorov--Smirnov distribution, which is the distribution of the
supremum of the Brownian bridge process $\sup_{0\leq \lambda\leq 1}\{|W(\lambda)-\lambda W(1)|:\ {0\leq \lambda \leq 1}\}$.


The cumulative sums test is based on partial sums and not robust to outlying observations.
 We propose a test based on the Hodges--Lehmann two-sample estimator of a location shift, which is the median pairwise difference $\med_{i=1,\ldots,k;j=k+1,\ldots,n} (X_i -X_j)$, instead of the mean difference. The Hodges--Lehmann estimator is robust, and its asymptotic efficiency relatively to the mean difference is larger than $95\%$ in case of Gaussian observations and never below $86\cddot 4\%$ in case of continuous distributions \citep{hodges1963}.  \citet{fried2011} explored the good robustness properties of two-sample tests based on this estimator, and
\citet{dehling2012} proved its asymptotic normality in the case of short range dependent observations.
The change-point test constructed here will be valid without any moment assumptions on the underlying data, and can thus be applied
to arbitrarily heavy-tailed data.

The Hodges--Lehmann change-point statistic proposed here is
\[
 M_n= {n}^{1/2}\max_{1\leq k \leq n} \frac{k}{n} \left(1-\frac{k}{n}\right)
 \left| \med\{(X_j-X_i):\ 1\leq i \leq k< j \leq n   \} \right|.
\]
$\Hy_0$ will be rejected for large positive values of $M_n$.
The asymptotic distribution of $M_n$ can be derived from studying the process
$[{n}^{1/2} \lambda (1-\lambda) \med\{(X_j-X_i):\ 1\leq i \leq \lfloor n\lambda\rfloor, \lfloor n\lambda\rfloor+1\leq j \leq n\},\ 0\leq \lambda \leq 1]$. More generally, we study the quantile process of the values
\[
  g(X_i,X_j), \quad 1\leq i\leq \lfloor n\lambda\rfloor,\;  \lfloor n\lambda\rfloor+1 \leq j \leq n,
\]
indexed by $0\leq \lambda \leq 1$, where $g(x,y)$ is a given function of two variables.
We investigate the asymptotic distribution of this process in the case of short range dependent data, and obtain the following theorem as a special case of the more general Theorem \ref{th:2uq-fclt}.

\begin{theorem}
Let $(Y_i:\ {i\in \Z})$ be a stationary process that is a near epoch dependent functional of an
absolutely regular process $(Z_i:\ {i\in \Z})$ with mixing coefficients $(\beta_j: \ j\in\N)$ and
approximating constants $(a_j:\ j\in \N)$ satisfying $\beta_j=O(j^{-8})$ and $a_j=O(j^{-12})$.
Moreover, let $Y_1$ have an absolutely continuous distribution with density $f(x)$ and assume that $
u(x)=\int f(y) f(x+y) dy$ is $\frac{1}{2}$-H\"older continuous. Then, under $\Hy_0$, we obtain
\[
{n}^{1/2}\max_{1\leq k \leq n} \frac{k}{n} (1-\frac{k}{n})
\left| \med\{(X_j-X_i):1\leq i \leq k< j \leq n   \}\right|
 \claw \frac{\sigma}{u(0)} \sup_{0\leq \lambda \leq 1}
\left|W^{(0)}(\lambda)\right|,
\]
in distribution, where $\{W^{(0)}(\lambda): \ {0\leq \lambda \leq 1}\}$ denotes a standard Brownian bridge process,
\[
\sigma^2=\sum_{i=-\infty}^\infty \Cov\{F(X_0),F(X_i)\},
\]
and $F$ is the marginal distribution function of the $X_i$, $i\in\Z$.
\label{th:HL-process}
\end{theorem}

For a definition of the mixing coefficients and the approximation constants see the next section.
By combining these two notions of weak dependence, our assumptions cover essentially all classes of short range dependence processes.
Long range dependence, arising for instance in  $MA(\infty)$-processes with non-summable coefficients, is not covered by our assumptions, and indeed our results do not apply to long range dependent processes.
\citet{borovkova2001} provide a detailed list of examples of short range dependent processes covered by our theory. For instance, stationary autoregressive moving average processes have an exponential decay of the approximation constants $a_j$ and the mixing coefficients $\beta_j$ are zero for $j\geq 1$.
The approximation constants $a_j$ for GARCH(1,1)-processes also converge to 0 exponentially fast \citep{hansen1991}. More general nonlinear time series models often can be represented as a Volterra sequence
\[
X_i=\sum_{l=1}^\infty\sum_{u_1,\ldots,u_l=0}^\infty g_l(u_1,\ldots,u_l)Z_{i-u_1}\ldots Z_{i-u_l}.
\]
and fulfill the conditions of our theorem if some summability condition hold for the coefficients $g_l(u_1,\ldots,u_l)$ and absolute regularity holds for the sequence $(Z_i: \ i\in\Z)$. Furthermore, many dynamical systems are covered by our assumptions, see \citet{borovkova2001}. Our results are general and can be applied without choosing and fitting specific time series models.

Application of Theorem \ref{th:HL-process} needs consistent estimators for the nuisance parameters
$\sigma$ and $u(0)$. In order not to restrict our analysis to a specific time series model, we use nonparametric estimators. We use overlapping subsampling for estimation of $\sigma$,
\[
 \hat{\sigma}_n = \frac{{\pi}^{1/2}}{ (2l)^{1/2}(n-l+1)} \sum_{i=0}^{n-l} |\sum_{j=i+1}^{i+l} \{F_n(X_j)-0\cddot 5\}|,\]
where $F_n$ is the empirical distribution function of $X_1,\ldots,X_n$.
\citet{dehling2013} have established consistency of the non-overlapping version of this estimator under the same assumptions as made here if the block length $l=l_n\rightarrow\infty$ fulfills $l_n=o(n^{-1/2})$. Consistency of $\hat{\sigma}_n$ under the hypothesis can be shown similarly. To achieve consistency under the alternative of one level shift, we can split the time series into three disjoint subsequences of similar length and use the median of the resulting three separate estimations. This estimates $\sigma$ consistently also under the alternative. However, according to our experience gained in simulations this splitting should only be applied if each subsequence consists of about 200 or more observations, since the median of the right-skewed and downward biased individual estimations is strongly downward biased otherwise.

For the parameter $u(0)$, observe that $u(x)$ is the density of $X-Y$, where $X$ and $Y$ are independent random variables with the same distribution as $X_1$. An estimator of $u(0)$ can be constructed applying a kernel density estimator to the pairwise differences $X_i-X_j$, $1\leq i<j\leq n$, leading to
\[
 \hat{u}(0) = \frac{2}{n(n-1)b} \sum_{1\leq i<j\leq n} K\left(\frac{X_i-X_j}{b}\right),
\]
for a symmetric, Lipschitz-continuous kernel function $K$ which integrates to 1. Below, we show that $\hat{u}(0)$ is a consistent estimator of $u(0)$ under $\Hy_0$, provided that the bandwidth $b=b_n$ is chosen appropriately. According to our experience gained from simulations, we recommend
estimation of $u(0)$ from data sets which are corrected for a possible level shift at each $k \in \{1,\ldots,n-1\}$. For this, we subtract the median pairwise difference $\med\{(X_j-X_i): {1\le i\le k<j\le n}\}$ from $X_{k+1},\ldots,X_n$ when considering the possibility of a shift at time $k$.
Application of $\hat{u}(0)$ to the corrected data $X_1,\ldots,X_{k},X_{k+1}^{(k)},\ldots,X_n^{(k)}$ leads to different estimates $\hat{u}_{k,n}(0)$ and a more powerful test. In the online supplement we prove that the difference $\hat{u}_{k,n}(0)-\hat{u}(0)$ is asymptotically uniformly (with respect to $k$) negligible under $\Hy_0$. Multiplication by $\hat{u}_{k,n}(0)$ thus means consistent scaling under $\Hy_0$ and also at the true position $k^\star$ of a single level shift as its effect in the corrected estimate $\hat{u}_{k^\star,n}(0)$ cancels out.

 The following corollary states that the change-point test statistic proposed in this paper follows  asymptotically a Kolmogorov--Smirnov distribution under $\Hy_0$, like the cumulative sums test statistic, for the estimates of $\sigma^2$ and $u(0)$ discussed before.

\begin{corollary}\label{cor:HL-test}
Under the same assumptions as in Theorem \ref{th:HL-process}, and applying subsampling estimation of $\sigma$ and kernel density estimation of $u(0)$ using block lengths $l_n=o(n^{-1/2})$ and bandwidths $b_n=o(1)$ such that $nb_n^4\to \infty$,
we obtain that the test statistic
\[
T_n=  \frac{{n}^{1/2}}{\hat{\sigma}_n} \max_{1\leq k \leq n} \hat{u}_{k,n}(0)\frac{k}{n} \left(1-\frac{k}{n}\right)
\left| \med\{(X_j-X_i):1\leq i \leq k, k+1\leq j \leq n   \}\right|
\]
converges in distribution to $\sup_{0\leq \lambda \leq 1} \left| W^{(0)}(\lambda)\right|$ under $\Hy_0$, where
$\{W^{(0)}(\lambda):\ 0\le \lambda\le 1\}$ denotes a standard Brownian bridge process.
\end{corollary}

So we can reject the hypothesis if the value of the test statistic $T_n$ exceeds $q_{1-\alpha}$, where $q_{1-\alpha}$ is the $(1-\alpha)$-quantile of the Kolmogorov-Smirnov distribution. This leads to a test with asymptotical level $\alpha$. This test is consistent under fixed alternatives. To see this, we consider fixed values of $\tau\in (0,1)$ and $\Delta\neq 0$ and the sequence of alternatives
\[
  A_n=A_n(\tau,\Delta):\; \mu_1=\ldots=\mu_{[n\tau]}=\mu_{[n\tau]+1}-\Delta=\ldots=\mu_n-\Delta.
\]

\begin{theorem}\label{theo:consistent}
Under the above sequence of alternatives $A_n$ and the assumptions of Theorem 1, $T_n$ converges to infinity in probability as the sample size $n$ increases. 
\end{theorem}


\section{Main Mathematical Results}
\subsection{Near Epoch Dependent Processes}
We derive the asymptotic results in this paper under the assumption of short range dependence.
In the literature, there is a wide range of notions for this.
    We follow an approach used already by \citet{billingsley1968} and \citet{ibragimov1971} and
    assume that the noise process $(Y_i:\ {i\in \Z})$ is near epoch dependent on an absolutely regular process.

\begin{definition}
(i) Let $\A,\B\subset \F$ be two $\sigma$-fields on the probability space $(\Omega,\F,P)$. We define the absolute regularity coefficient $\beta(\A,\B)$ by
\[
 \beta(\A,\B)= E \{\sup_{A\in \A} | \pr(A|\B)-\pr(A) | \}.
\]
(ii) For a stationary process $(Z_i: \ {i\in \Z})$ we define the absolute regularity coefficients
\[
 \beta_j=\sup_{i\geq 1} \beta(\G_1^i,\G_{i+j}^\infty),
\]
where $\G_k^l$ denotes the $\sigma$-field generated by the random variables $Z_k,\ldots, Z_l$. The process $(Z_i:\ {i\in \Z})$ is called absolutely regular if $\beta_j\rightarrow 0$ as $j\rightarrow \infty$.
\\[1mm]
(iii) Let $\{(X_i,Z_i):\ {i\in \Z}\}$ be a stationary process. We say that $(X_i:\ {i \in\N_0})$ is  $L^1$-near epoch dependent on $(Z_i:\ {i\in \Z})$ with approximating constants $(a_j:\ {j\in\N})$, if $\lim_{j\rightarrow \infty }a_j=0$ and
\[
 E\{|X_0-E(X_0|\G_{-j}^j)|\}  \leq a_j.
\]
\end{definition}

\subsection{Two-sample empirical U-quantile process}
   Now we investigate the two-sample empirical quantile process associated with the kernel $g(x,y)$.  We formally define this process, as well as the related two-sample empirical U-process, both in a slightly more general setup of empirical processes indexed by classes of functions.

\begin{definition}
Let $h:\R^2\times \R \rightarrow [0,1]$ be a measurable function, and let $(X_i:\ {i\in\Z})$ be a stochastic process.
\\[1mm]
(i) We define the two-sample empirical U-process
\[
  U_n(\lambda,t)=\frac{1}{\lfloor n\lambda\rfloor (n-\lfloor n\lambda\rfloor)}\sum_{i=1}^{\lfloor n\lambda\rfloor} \sum_{j=\lfloor n\lambda\rfloor+1}^n
h(X_i,X_j,t), \quad 0\leq \lambda\leq 1,\; t\in \R.
\]
(ii) Given $p \in [0,1]$, we define the two-sample empirical U-quantile process
\[
  Q_n(\lambda,p)=\inf\{t:U_n(\lambda,t)\geq p\},\quad 0\leq \lambda \leq 1.
\]
\end{definition}

\begin{remark}
{\rm
(i) Given a kernel $g(x,y)$, we define $h(x,y,t)=1\{g(x,y)\leq t \}$.
Then, $U_n(\lambda,\cdot)$ is the empirical distribution function of the data $g(X_i,X_j)$, $1\leq i\leq \lfloor n\lambda\rfloor < j\leq n$, and $Q_n(\lambda)$ is the $p$-th quantile of the same data.\\
(ii) For fixed $t$, the process $\{U_n(\lambda,t):\ {0\leq \lambda \leq 1}\}$ is a two-sample U-process that has been introduced and investigated by \citet{dehling2015}.
}
\end{remark}

A useful tool for analyzing the asymptotic distribution of the two-sample empirical U-quantile process
$\{Q_n(\lambda):\ {0\leq \lambda \leq 1}\}$ is the following Hoeffding decomposition.

\begin{definition}
Let $h(x,y,t)$ be a measurable function, and let $X,Y$ be independent random variables with the same distribution as $X_i$. Then we define the functions $U(t), h_1(x,t)$, and $h_2(y,t)$ by
\begin{eqnarray}
U(t)&=&E\{h(X,Y,t)\}\label{eq:u_t},\\
h_1(x,t)&=& E\{h(x,Y,t)\}-U(t) \label{eq:h_1},\\
h_2(y,t)&=&E\{h(X,y,t)\}-U(t) \label{eq:h_2}.
\end{eqnarray}
Moreover, let $Q(p)=\inf\{t:U(t)\geq p\}$ be the quantile function and $t_p=Q(p)$ the $p$-quantile.
\end{definition}

Our theorems will require the following technical conditions regarding the process $(X_i:\ {i\geq 1})$ and the kernel $h(x,y,t)$.

\begin{assumption}
(C1) The process $(X_i:\ {i\in\Z})$ is a near epoch dependent functional of an absolutely regular process $(Z_i:\ {i\in \Z})$ with mixing coefficients $(\beta_j:\ j\in \N)$ and approximation constants $(a_j:\ {j\in \N})$, such that for some constant $\beta>3$ we have
\begin{eqnarray*}
\beta_j=O(j^{-\beta}),&&\quad  a_j=O(j^{-(\beta+3)}).
\end{eqnarray*}
(C2) The function $U(t)$, as defined in (\ref{eq:u_t}), is differentiable in a neighborhood of $t_p$. Moreover, $u(t)=U^\prime(t)$ satisfies $u(t_p)>0$, and,
as $t\rightarrow t_p$,
\[
 |U(t)-p-u(t_p)(t-t_p)|=O(|t-t_p|^{3/2}).
\]
(C3) The kernel $h:\R^3\times \R$ is a bounded measurable function. Moreover, $t\mapsto h(x,y,t)$ is nondecreasing, and $(x,y)\mapsto h(x,y,t)$ is
uniformly $1$-Lipschitz continuous in a neighborhood of $t_p$. This means that there exists a neighborhood of $t_p$ and a constant $L>0$ such that
\begin{eqnarray*}
 E\left[|h(X,Y,t) -h(X^\prime,Y,t)|1{\{|X-X^\prime| \leq \epsilon \}}  \right]&\leq& L\, \epsilon, \\
  E\left[|h(X,Y,t) -h(X,Y^\prime,t)|1{\{|Y-Y^\prime| \leq \epsilon \}}  \right]&\leq & L\, \epsilon
\end{eqnarray*}
holds for all $t$ in this neighborhood, for all $\epsilon >0$, and for all quadruples $X,Y,X^\prime,Y^\prime$ of random variables such that $(X,Y)$ has joint distribution $\pr_{X_1}\times \pr_{X_1}$ or $\pr_{X_1,X_k}$, for some $k$, and such that $X^\prime$ and $Y^\prime$ each have the same marginal distribution as $X_i$.
\end{assumption}

\begin{theorem}
Let $\{X_i:\ {i\in\Z}\}$ be a near epoch dependent functional of an absolutely regular process such that assumption (C1) is satisfied and  $h:\R^3\times \R\rightarrow \R$ a measurable kernel such that assumptions (C2) and (C3) hold. Then we have the following convergence in distribution:
\[
\left[{n}^{1/2}\lambda (1-\lambda)\{Q_n(\lambda,p)-Q(p)\}: {0\leq \lambda \leq 1}\right]
 \claw \left[(1-\lambda) W_1(\lambda) +\lambda \{W_2(1)-W_2(\lambda)\}: {0\leq \lambda\leq 1}\right],
\]
where $\{W_1(\lambda),W_2(\lambda)\}$ is a two-dimensional Brownian motion with covariance structure
\[
 \Cov\{W_i(\mu),W_j(\lambda)\}=(\mu\wedge \lambda)\frac{1}{u^2\{Q(p)\}}
 \sum_{k\in \Z} E[h_{i}\{X_0, Q(p)\}, h_{j}\{X_k,Q(p)\}].
\]
\label{th:2uq-fclt}
\end{theorem}

An important ingredient in the proof of the limit theorem for the two-sample $U$-quantile process is the Bahadur--Kiefer representation of the $U$-quantiles, see \citet{bahadur1966}. The Bahadur-representation for two-sample $U$-quantiles (with fixed $\lambda$) has been studied by \citet{inagaki1973} for independent data and by \citet{dehling2012} for dependent data. To the best of our knowledge, there are no results for the process indexed by $\lambda$. There is much more literature on one-sample $U$-quantiles, beginning with \citet{geertsema1970}. In this case, better rates of the Bahadur representation are known, see \citet{wendler2011}.

\begin{theorem}
Under the same assumptions as in Theorem~\ref{th:2uq-fclt}, we obtain
\[
 \sup_{0\leq \lambda \leq 1}\!
  \lambda (1-\lambda) \left[\! Q_n(\lambda,p)-Q(p) +
 \frac{U_n\{\lambda,Q(p)\}-p}{u\{Q(p)\}}\!  \right] = O_{\pr}(n^{-\frac{5}{9}}).
\]
\label{th:bahadur}
\end{theorem}

\section{Simulation Results}

We illustrate the practical value of the theoretical results presented above in a simulation study using the statistical software R \citep{R2016}. Time series are generated from autoregressive moving average models,
\[Y_i=\phi_1Y_{i-1}+\phi_2Y_{i-2}+\theta\epsilon_{i-1}+\epsilon_i,\]
 where $(\epsilon_i:\ i\in\Z)$ are Gaussian, $t_3$- or $\chi_3^2$-distributed white noise innovations, scaled to have zero mean and unit variance and representing normal tails, heavy tails and skewness, respectively. We consider a broad range of practically relevant positive autocorrelation structures, namely first order models with $\phi_1=0,0\cddot 4,0\cddot 8$ or $\theta=0\cddot 5,0\cddot 8$, second order models with $(\phi_1,\phi_2)=(0\cddot 4,0\cddot 3)$, and mixed models with $(\phi_1,\theta)= (0\cddot 3,0\cddot 5)$, setting the other parameters to zero.

Besides the cumulative sum test, we compare our test to  further competitors also designed for shift detection in weakly dependent data. Extending work by \citet{deJong2000}, \citet{huskova2012} suggest a version of the cumulative sum test based on the partial sums of M-residuals $\psi\{Y_i-\hat{\mu}_n(\psi)\}$, replacing the sign function used by the former authors by the Huber function $\psi(x)=x 1(|x|\le c\hat{\kappa}_n)+c\hat{\kappa}_n1(|x|>c\hat{\kappa}_n)$, where $\hat{\kappa}_n$ is a robust estimate of the standard deviation of the observations.
The Huber function comprises the sign and the identity function as limiting cases as the tuning constant $c\in [0,\infty)$ approaches zero or infinity, respectively. We call this the Huberization test, using the sample median $\hat{\mu}_n$ and median absolute deviation about the median $\hat{\kappa}_n$ of the data for standardization, and $c=1\cddot 5$ to achieve a reasonable  compromise between performance under normality and under heavy-tails.

\citet{dehling2015} construct a Wilcoxon change-point test based on the two-sample Wilcoxon--Mann--Whitney statistic. \citet{vogel2017} use the difference between the one-sample Hodges--Lehmann estimators
$\med\{(X_i+X_j)/2:i=1,\ldots,j-1,j=2,\ldots,l\}$ obtained from the first $l=k$ and all $l=n$ observations
to construct a robust alternative to the cumulative sum test, called one-sample Hodges--Lehmann change-point test hereafter.

For estimation of the respective long-run variances needed for standardization of the test statistics, we follow \citet{carlstein1986} and
  \citet{dehling2013}, who suggest overlapping subsampling with Carlstein's adaptive block length
$l_n=\max[\lceil n^{1/3}\{2\hat{\phi}/(1-\hat{\phi}^2)\}^{2/3}\rceil,1]$ 
depending on the lag-one sample autocorrelation $\hat{\phi}$.
In case of the Wilcoxon and the Hodges--Lehmann change-point tests we use Spearman's rank autocorrelation for this, while it is the ordinary sample autocorrelation of the possibly transformed data in case of the ordinary and the Huberization cumulative sums test.
For comparison we also comment on a fixed block length $l= \lceil (3n)^{1/3}+1\rceil$, which is appropriate for a first order autoregression with $\phi=0\cddot 5$ according to Carlstein's rule. The tests based on the one- or two-sample Hodges-Lehmann estimators additionally need estimation of a density at its median which is done as discussed before Corollary 1.

First we assess the sizes of the tests, applying them with the asymptotic critical value 1$\cddot$36 at a nominal significance level of 5\%. We generate 6000 time series for each of the 21 combinations of an innovation distribution and an autoregressive moving average  model mentioned above, considering time series of lengths $n=100$ and $n=200$. The standard errors of the estimated rejection rates are about 0$\cddot$3\%. Figure \ref{fig:sizes} presents dotplots of the empirical sizes for the different settings.
The adaptive block length works well also in the presence of an additional moving average component but leads to oversized tests for the second order autoregression. These problems with the size are larger if $n=100$, where only the Wilcoxon change-point test keeps its nominal size for almost all scenarios. Increasing the sample size to $n=200$ reduces the problems for the tests based on the Hodges--Lehmann estimators.
Subsampling with the fixed block length (not shown here) leads to substantially oversized tests in case of strong positive autocorrelations.


\begin{figure}[htbp!]\centering
\resizebox{8.5cm}{8.5cm}{\includegraphics{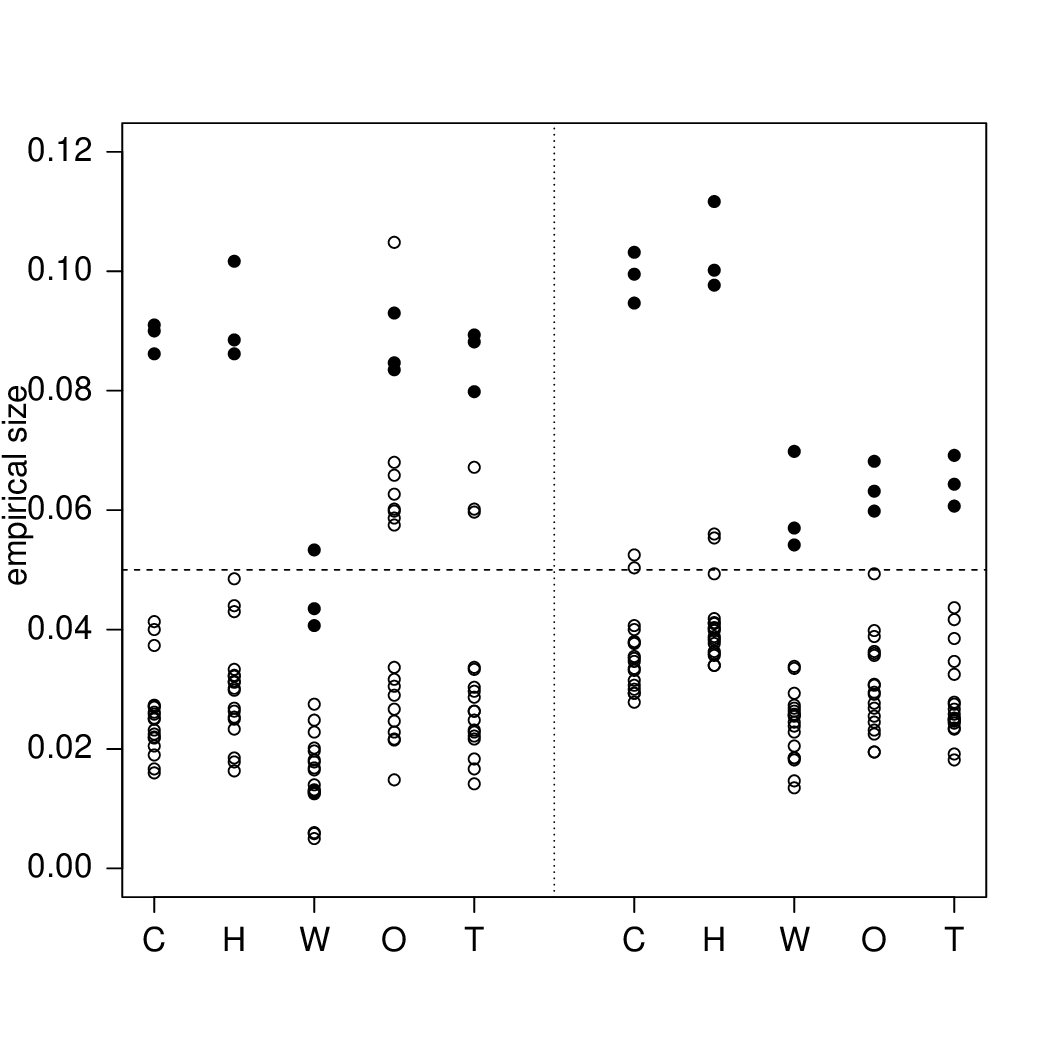}}\\ 
\caption{\label{fig:sizes}Empirical sizes in percent for time series of length $n=100$ (left) and $n=200$ (right) generated from different time series models with normal, heavy-tailed or skewed innovations.
Cumulative sums test (C) and its variants using Huberization (H), the Wilcoxon (W), one- (O) or two-sample (T) Hodges--Lehmann statistic.
Subsampling with an adaptive block length assuming a first order autoregression keeps the nominal significance level 5\% well also in the presence of a moving average part if $n=200$, but has some difficulties with the second order autoregression represented by filled dots, particularly for H and C.
 }
\end{figure}

Next we compare the power of the tests after size-correction to achieve a fair comparison,
 using the respective 95\% percentiles obtained in the simulations under $\Hy_0$ for the same scenario as critical values.
The power of the size-corrected tests is examined by generating 600 time series of length $n=100$ or $n=200$ from each model and each of ten different shift heights and two change-point positions $\tau=\delta n$, $\delta=0\cddot5$ or $\delta=0\cddot75$. Larger heights are considered for stronger positive autocorrelations and for change-points outside the center of the time series, since detection of small shifts becomes more difficult then. Table \ref{tab:heights} reports the shift heights which were chosen after some experiments such that the more powerful tests achieve powers of about 95\% in case of the highest shifts considered.

\begin{table}
\caption{
Smallest height of a shift considered for each model. These heights are multiplied by $1,\ldots,10$ to achieve increasingly large shifts. These heights are multiplied by 1$\cddot$5 if the shift is not in the center but after 75\% of the observations, and additionally by 1$\cddot$4 if $n=100$}{
\begin{tabular}{cccccccccc}
$\phi_1,\phi_2,\theta$ & $0,0,0$ & 0$\cddot$4,0,0& 0$\cddot$8,0,0 & 0$\cddot$4,0$\cddot$3,0 & 0$\cddot$4,0,0$\cddot$5 & 0,0,0$\cddot$5 &
0,0,0$\cddot$8\\
Heights  & 0$\cddot$06 & 0$\cddot$1 & 0$\cddot$25 & 0$\cddot$2 & 0$\cddot$13 & 0$\cddot$1 & 0$\cddot$12
\end{tabular}}\label{tab:heights}
\end{table}

Figure \ref{fig:powercurves} depicts estimated power curves of the size-adjusted tests with adaptive subsampling for shifts of increasing height in independent observations.
 The tests perform quite similar if a shift occurs in the center of independent Gaussian data. If a shift occurs outside the center, the Wilcoxon change-point test gets worse. In case of heavy-tailed $t_3$- or skewed $\chi_3^2$-distributed observations the ordinary cumulative sums test loses a lot of its power. In case of $\chi_3^2$-observations this also applies to the Huberization and even more to the one-sample Hodges--Lehmann change-point test. Huberization transforms the data symmetrically, which is not the best solution if the data are skewed. The problems of the one-sample Hodges-Lehmann change-point test can be explained by the lack of efficiency of the underlying estimator as compared to the mean or the median difference in case of skewed distributions  \citep{hoyland1965}.

\begin{figure}[htbp!]\centering
\resizebox{15.5cm}{5.8cm}{\includegraphics{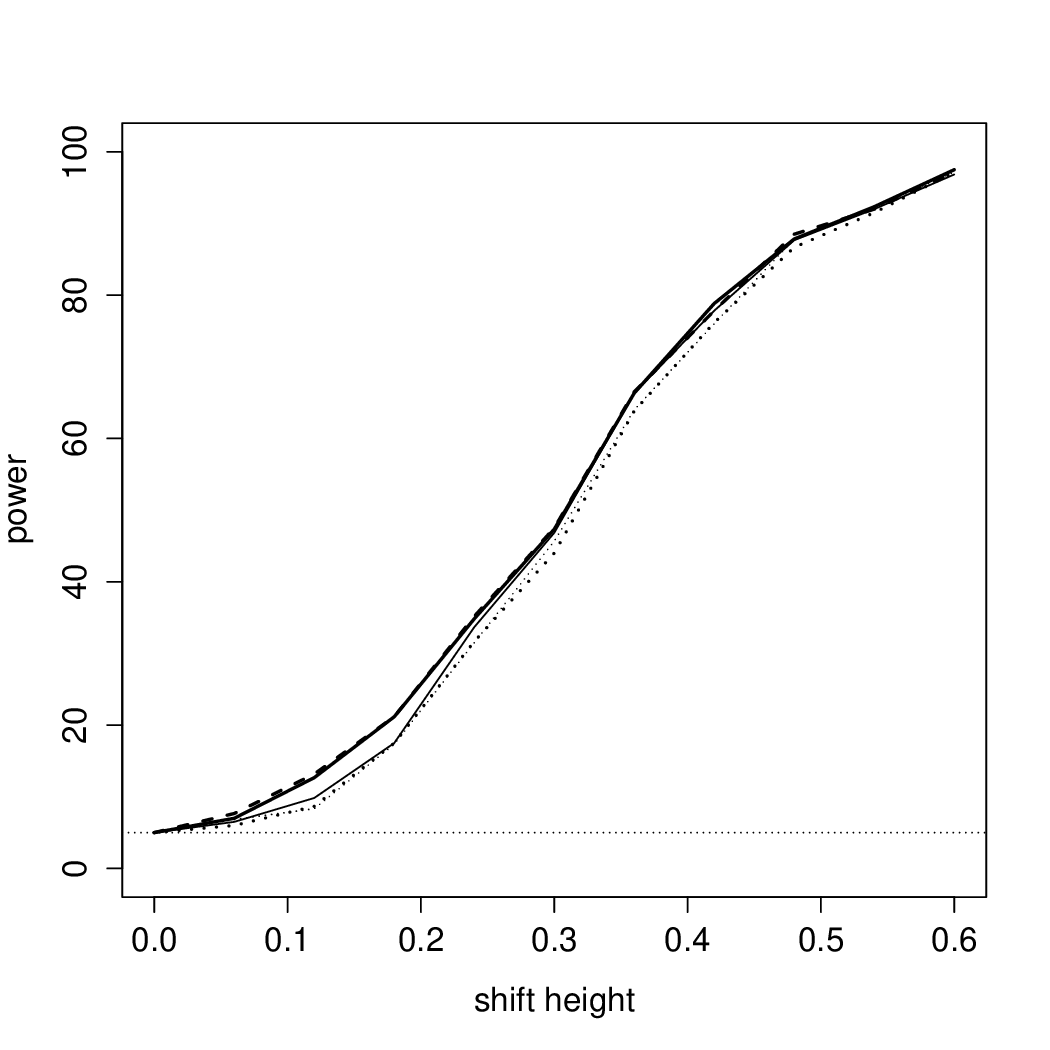}\includegraphics{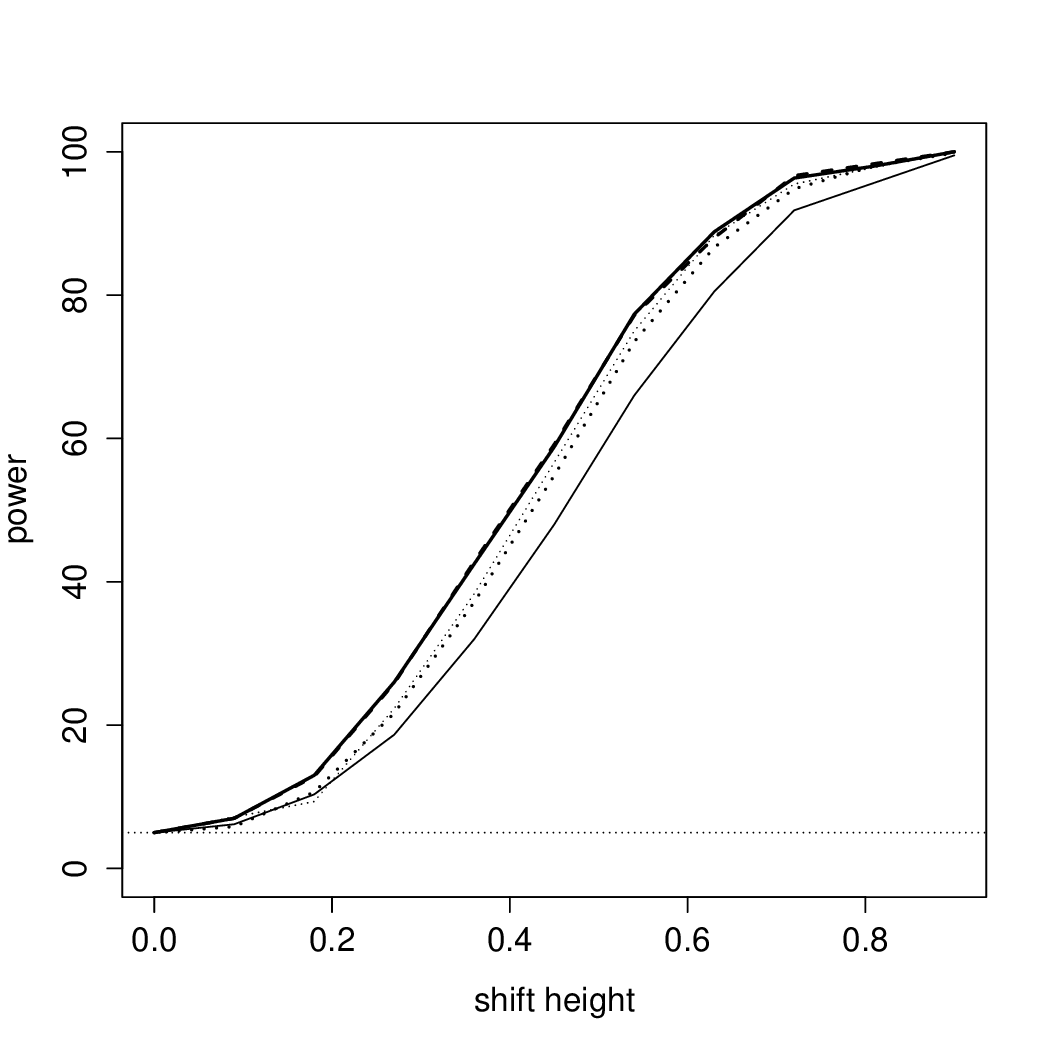}}\\
\resizebox{15.5cm}{5.8cm}{\includegraphics{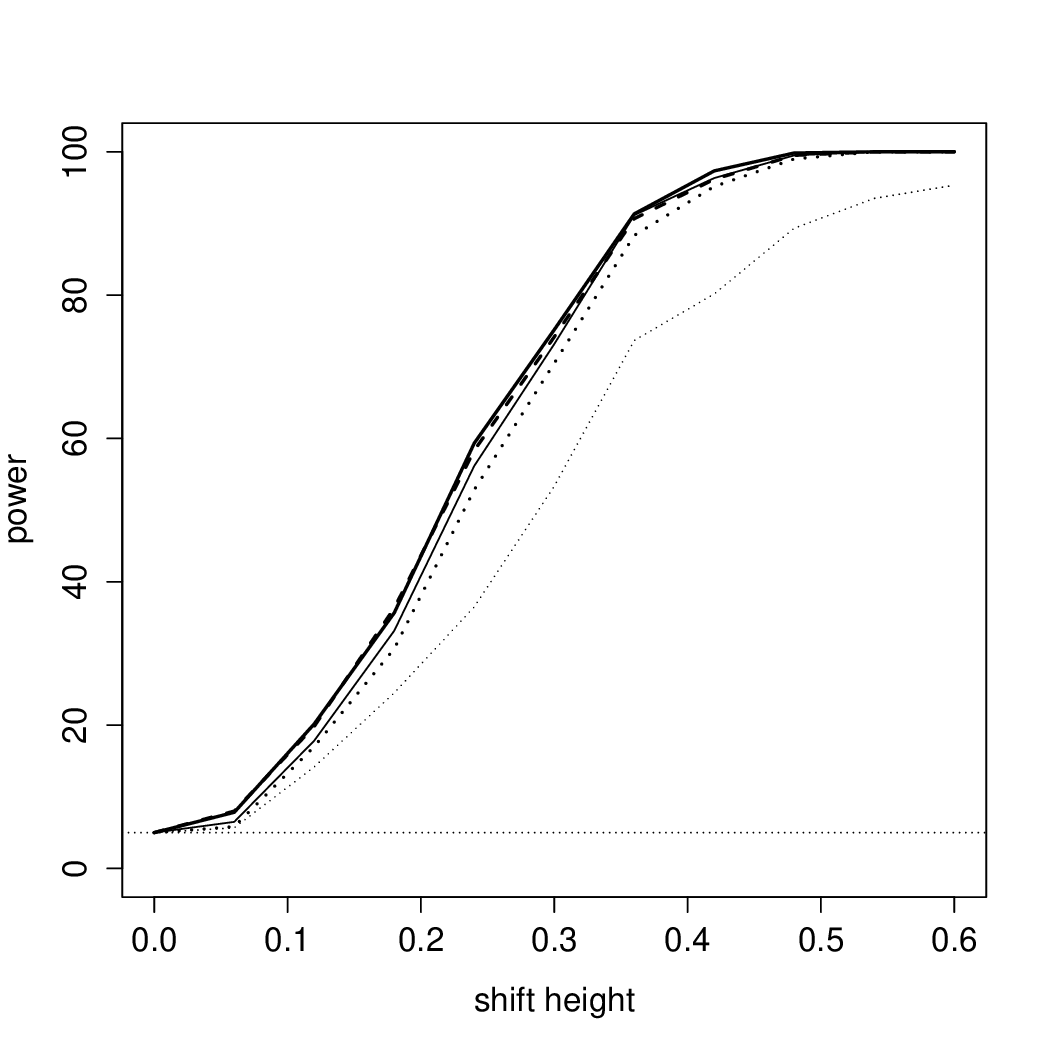}\includegraphics{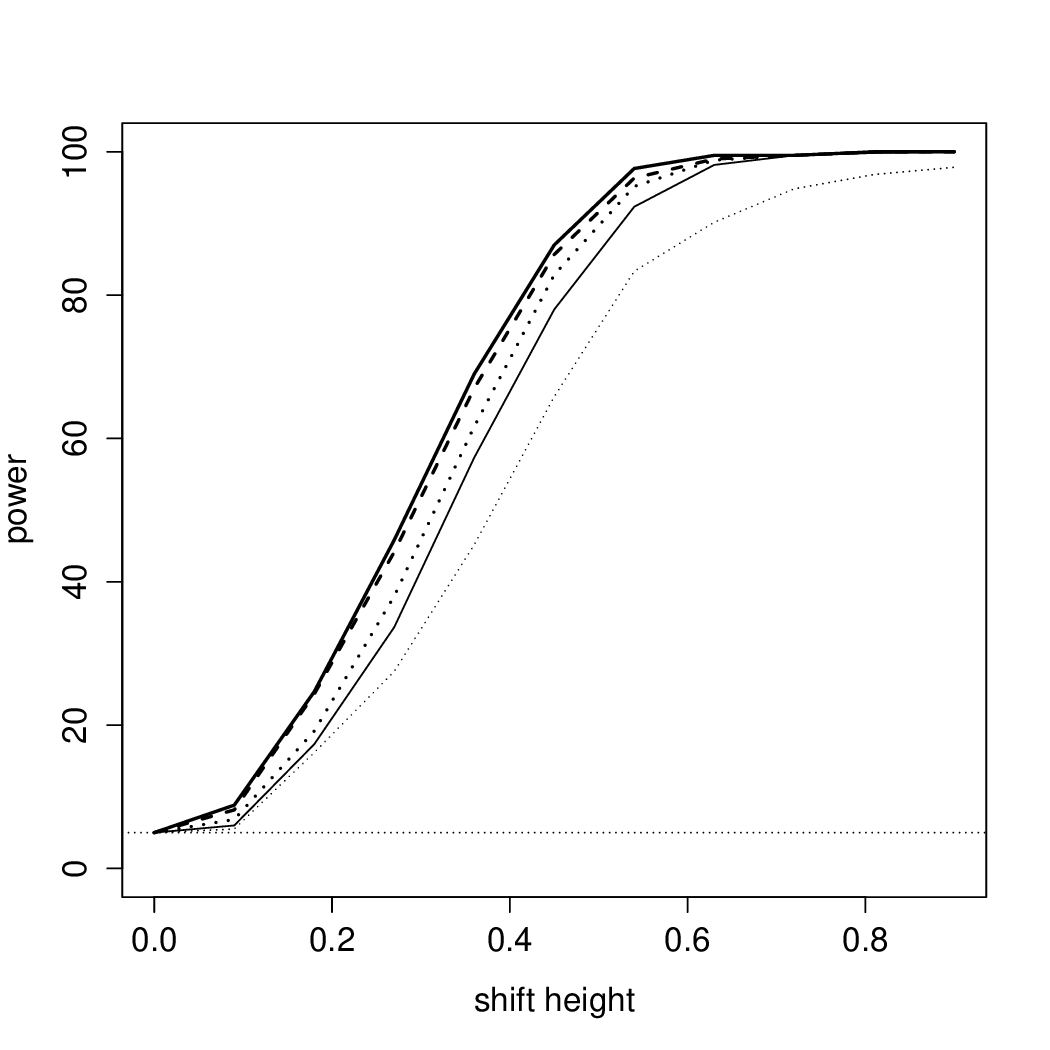}}\\
\resizebox{15.5cm}{5.8cm}{\includegraphics{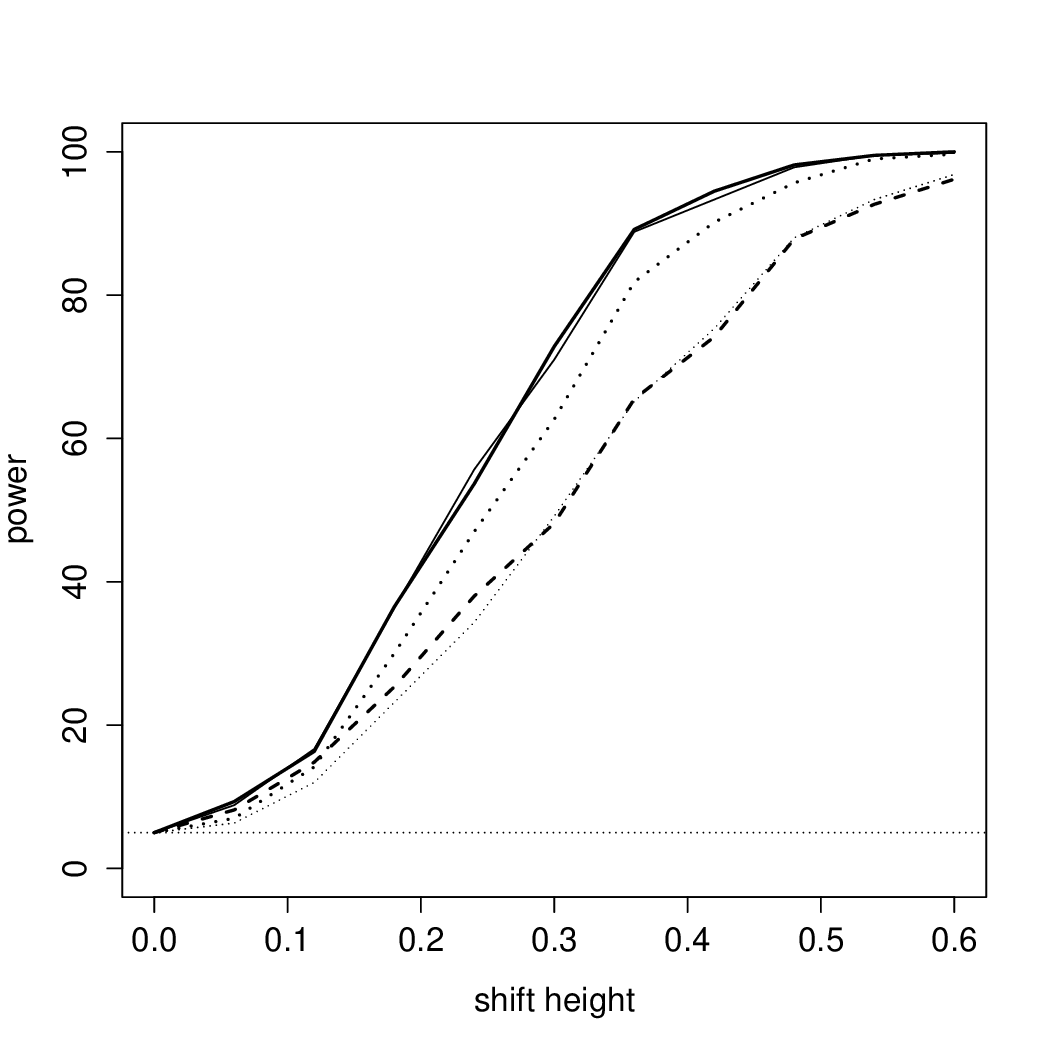}\includegraphics{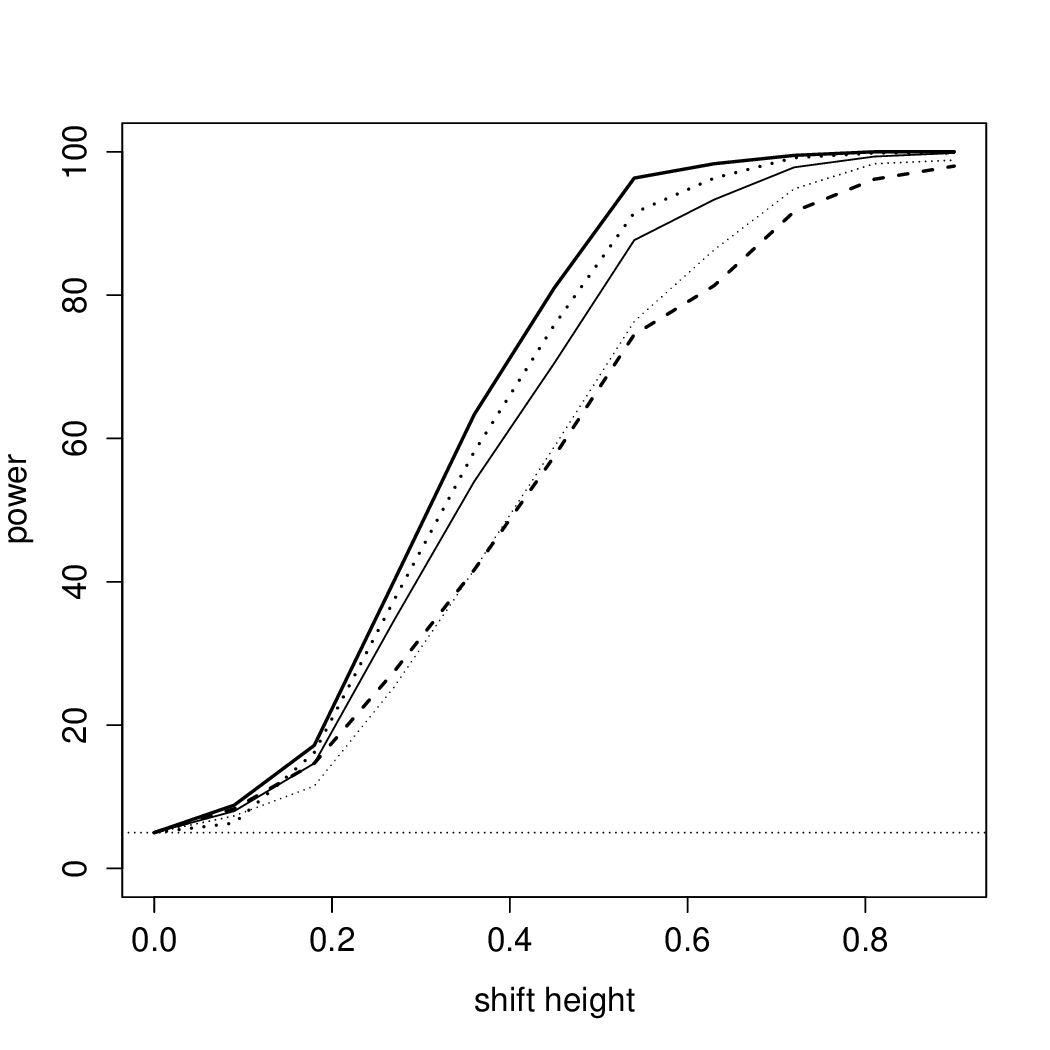}}\\
\caption{\label{fig:powercurves}Size-adjusted power in percent of the tests with adaptive subsampling in case of a shift after  100 (left) or 150 (right) out of  $n=200$ independent normally (top), $t_3$- (center) or $\chi_3^2$-distributed (bottom) observations. Cumulative sums test based on ordinary (thin dots) or Huberized observations (bold dots), on the Wilcoxon (thin solid), the two- (bold solid) or one-sample Hodges--Lehmann estimate (bold dashes). The pointwise standard errors of the power estimates are at most about 2\%.
 }
\end{figure}

We summarize the power curves  of the size-adjusted tests by calculating the average power across all shift heights for each model scenario. Figure \ref{fig:powers} shows dotplots of the average powers resulting for the different scenarios.
The ordinary cumulative sum test, the Wilcoxon and the Huberization change-point tests have some difficulties detecting shifts in data with strong positive autocorrelations, that is, the second order autoregression and the first order autoregression with $\phi_1=$0$\cddot$8, if there are only $n=100$ observations available. Only the tests based on the Hodges-Lehmann statistics do not show particular weaknesses, with the test based on the two-sample statistic advocated here providing the largest average powers.

\begin{figure}[htbp!]\centering
\resizebox{8.5cm}{8.5cm}{\includegraphics{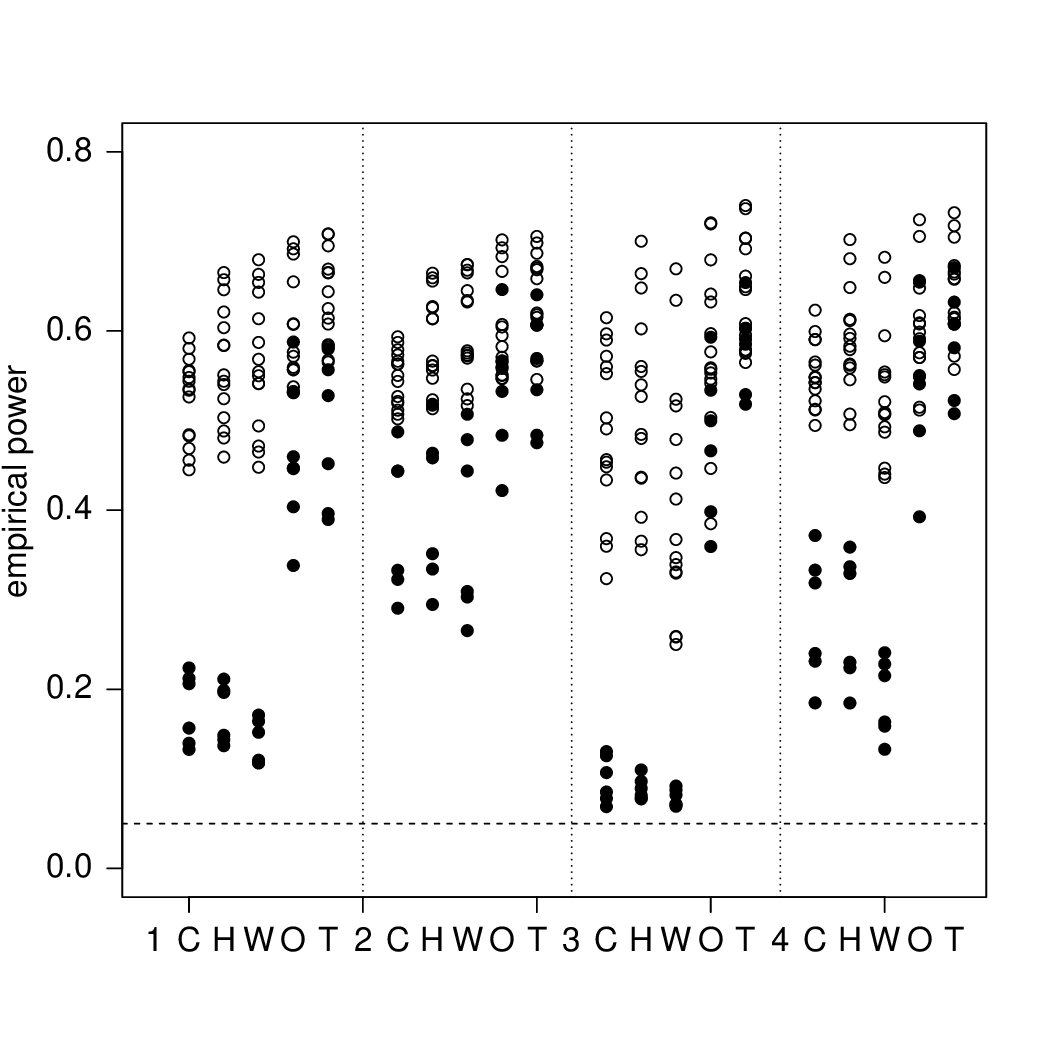}}\\
\caption{\label{fig:powers}Estimated average powers of the size-adjusted tests for different time series models and a shift in the center (panels 1 and 2) or after 75\% of the observations (panels 3 and 4) in a time series of length $n=100$ (panels 1 and 3) or $n=200$ (panels 2 and 4). Cumulative sums test (C) and its variants based on Huberization (H), Wilcoxon (W), one- (O) or two-sample Hodges--Lehmann (T) statistics. Results for the second and the first order autoregression with large positive autocorrelations are worse for some of the methods and depicted by filled black dots.    }
\end{figure}

For a more detailed comparison we perform an analysis of variance for the average size-adjusted powers of the tests.  Shifts inside and outside the center of the time series are analyzed separately, including main effects for the dependence structure and the innovation distribution but ignoring possible interactions. Table \ref{tab:power} in the online supplement reports the detailed results.

It turns out that the test based on the two-sample Hodges--Lehmann estimator is competitive to or better than the other tests in all scenarios considered here. The test based on the one-sample Hodges--Lehmann estimator performs similar to it for the symmetric innovation distributions but worse for the skewed $\chi_3^2$-distributed innovations. This agrees with the smaller asymptotic efficiency of the underlying estimator. The other tests perform somewhat less powerful than these in the presence of strong positive autocorrelations, particularly for the smaller sample size $n=100$. In case of the cumulative sums tests based on the ordinary or the Huberized observations this can be explained by the more difficult estimation of the long run variance, since for the tests based on the Hodges--Lehmann estimators it refers to bounded variables. The Wilcoxon change-point test is less powerful for a shift outside the center. The ordinary cumulative sums test seems to be inferior to the other tests for heavy-tailed $t_3$- or skewed $\chi_3^2$-distributed innovations and not much better for Gaussian innovations.
Additional simulations not reported here indicate that the advantage of the robust tests gets larger as the tails get heavier, see also \citet{huskova2012} and \cite{vogel2017}.


The time point $k$ where a test statistic as those considered here takes its maximum is a natural estimator of the time of a shift.
Figure \ref{fig:times} compares the resulting absolute estimation errors, relative to the length of the time series, obtained in case of the largest shift for each data generating process. Apparently, these estimators work similarly well here. Only the estimator based on the Wilcoxon statistic behaves somewhat differently as it provides more precise estimations if the shift is in the center, but less precise ones if it is far from the center.
Nevertheless, this estimator is consistent for the time of the change under conditions similar to those considered here as has been proven recently \citep{Gerstenberger2018}.

\begin{figure}[htbp!]\centering
\resizebox{8cm}{8cm}{\includegraphics{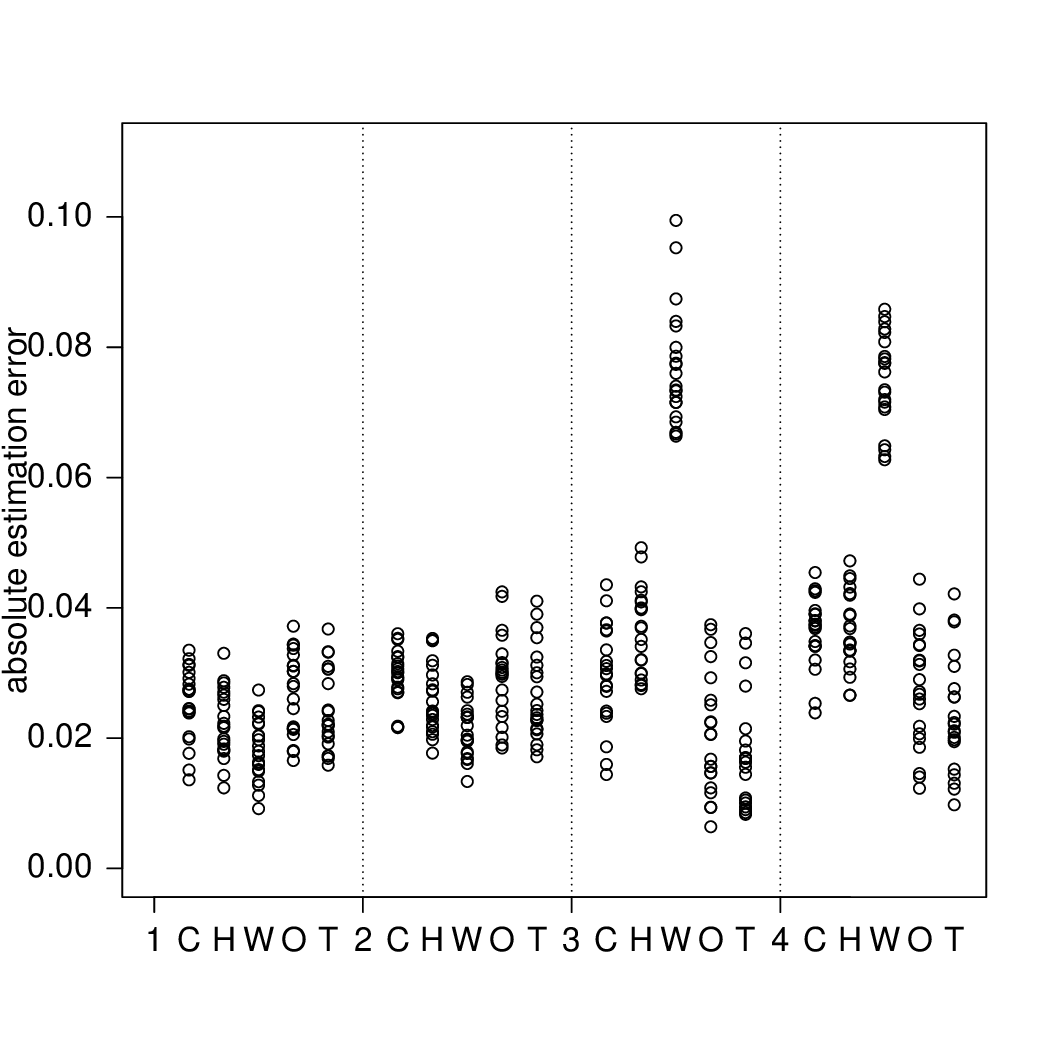}}\\
\caption{\label{fig:times}Average estimation errors of the time point of the change (relative to the length of the time series) for different time series models and a shift in the center (panels 1 and 2) or after 75\% of the observations (panels 3 and 4) in a time series of length $n=100$ (panels 1 and 3) or $n=200$ (panels 2 and 4). Cumulative sums test (C) and its variants based on Huberization (H), Wilcoxon (W), one- (O) or two-sample Hodges--Lehmann (T) statistics.     }
\end{figure}

\section{Data analysis}

For illustration we analyze the $254$ absolute daily stock returns of Volkswagen at the Frankfurt stock exchange in 2015, downloaded from yahoo finance. In September 2015 irregularities with the emissions of diesel cars of this company became public. Analyzing the stability of the level of the absolute values provides information about changes of the variability of the returns. A histogram and the sample autocorrelations of the absolute values indicate large skewness and positive dependencies with a lag-one autocorrelation of about 0$\cddot$6, but this might be due to change-points for instance in September, see Fig. \ref{fig:vw}. We perform several analyses with these data.

First we test the null hypothesis $\Hy_0^{n}:\mu_1=\ldots=\mu_n$ of a constant level of variability up to the $n$-th trading day, using the data for the first $n=30,\ldots,254$ trading days, only. All tests reject $\Hy_0^n$ at the usual nominal 0$\cddot$05 significance level if $n$ is between about 40 and 80,  possibly due to a larger variability in January than in February and March. However, some tests do not reject the null hypothesis $\Hy_0^{254}$ of a constant level throughout the year, probably due to difficulties with multiple shifts into opposite directions. All these tests are designed under the assumption of at most one change and such difficulties are well known for the cumulative sums test. As opposed to this, the two-sample Hodges--Lehmann change-point test constantly signals that the level is unstable whenever applied to $n\ge 40$ observations. This can be explained by the robustness of the underlying estimator of a level shift, which takes a large absolute value whenever there is a split of the $n$ data points such that the majority of the data before the split time point is at another level than the majority of the data after it. This applies similarly to the Huberization and, to a smaller extent, also to the Wilcoxon change-point test. The one-sample Hodges--Lehmann change-point test considered here compares the first $k$ to all $n$ observations so that the corresponding robust level estimates can be quite similar in case of a late shift.

The focus of our paper is on robust change-point tests within the at most one change scenario, but in practical applications like the one considered here several change-points can occur. A recent proposal for robust detection of multiple change-points in a sequence of independent observations is due to \citet{fearnhead2018}. A thorough solution of this problem for dependent data is outside the scope of this paper. However, we feel that simple strategies like binary segmentation can be improved by combination with a powerful robust test as proposed here. To illustrate this, we apply the tests to the whole year and estimate the date of the change-point by the value of the split $k$ for which the test statistic takes its maximum. Doing so, only our test dates a shift where we expect it to be, namely in September, while the Wilcoxon and the Huberization  test point at a change about two months earlier in July, and the other two tests do not detect a change at all. After detection of a change, the sample is split into the observations before and after the estimated change-point. Then the testing is repeated on each of these subsequences.
When applying this strategy with our Hodges--Lehmann change-point test, we sequentially detect further changes in July and at the beginnings of February and October, even when applying Bonferroni correction to achieve the same overall 0$\cddot$05 significance level in each step. The sample autocorrelations of the absolute returns within the different segments identified this way agree well with white noise assumptions.
 Among the other tests, only Huberization leads to an at least weakly significant test statistic in September in a second step when applied with Bonferroni correction, and to a third change-point in October without Bonferroni corection. The Wilcoxon test statistic detects a change in September only when being applied without Bonferroni correction.

These are just exploratory findings for a single data set, of course, and a careful examination needs further studies. Promising candidates for further improvements might be combinations of techniques like wild binary segmentation \citep{fryz2014} and the robust test statistic considered here.



\begin{figure}[htbp!]\centering
\resizebox{12cm}{8cm}{\includegraphics{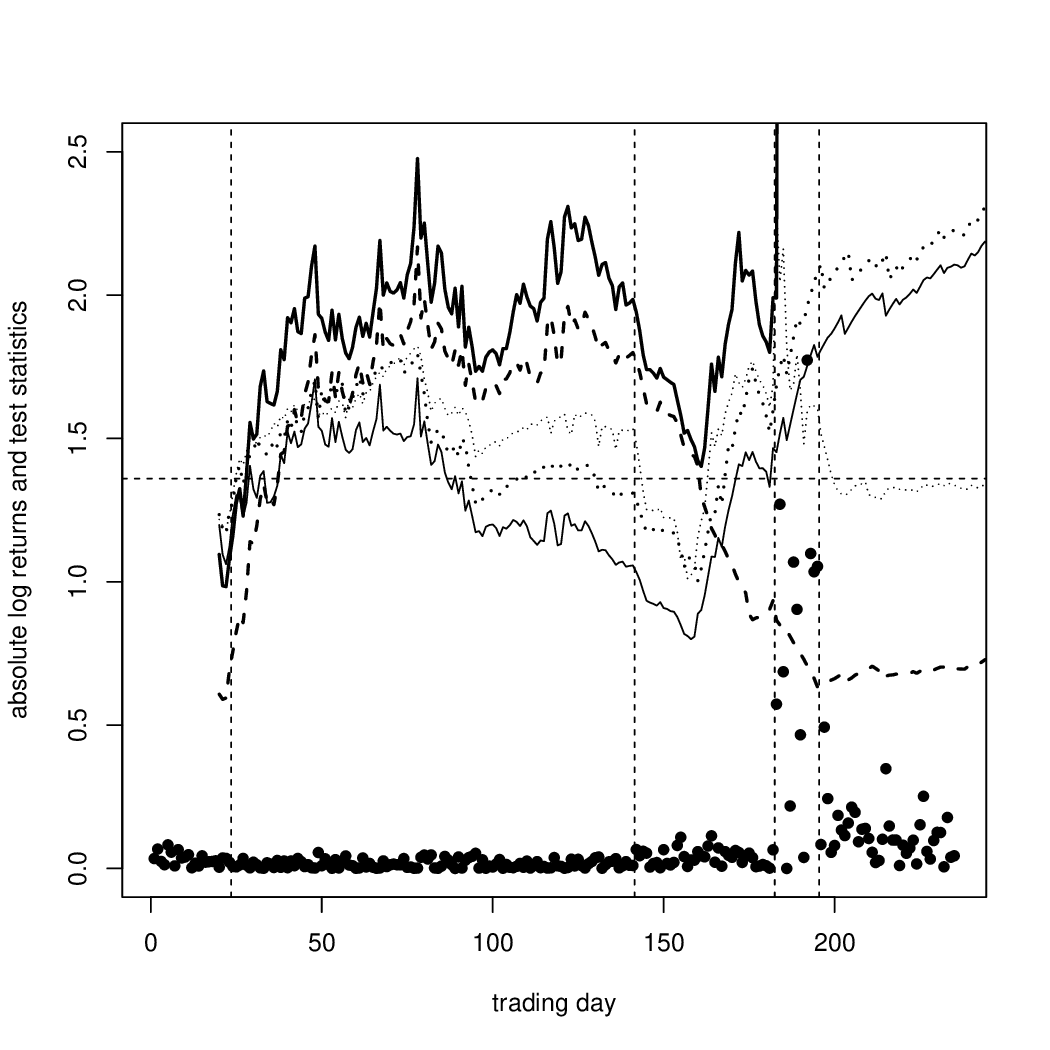}}
\caption{\label{fig:vw}Absolute daily log returns of VW in 2015 (filled dots) and change-point test statistics calculated from the first $n=20, \ldots,254$ observations:
cumulative sum test based on ordinary (thin dotted line) or Huberized observations (bold dotted line),  Wilcoxon (thin solid line), two- (bold solid line) or one-sample Hodges--Lehmann change-point test (bold dashed line). The horizontal dashed line marks the critical value at the nominal 5\% significance level. All tests reject the null hypothesis $\Hy_0^{n}:\mu_1=\ldots=\mu_n$ when being applied to the first $n\approx 50$ observations, probably due to a larger variability in January than in February and March. However, some tests do not reject the null hypothesis $\Hy_0^{254}$ of a constant level of variability throughout the year, probably due to difficulties with multiple shifts into opposite directions.
 The vertical lines mark the change-points detected by the test based on the two-sample Hodges-Lehmann estimator combined with binary segmentation.     }
\end{figure}

\section*{Acknowledgement}

This research was supported by the DFG Collaborative Research Center 823  {\em Statistical Modelling of Nonlinear Dynamic Processes}.


\newpage

\appendix

\section*{Appendix 1: auxiliary results}

The proofs require some further notations, which we introduce now. Given the kernel $h(x,y,t)$, we define the two-sample empirical $U$-process
\[
 U_{n_1,n_2}(t)=({n_1\, n_2})^{-1}\sum_{i=1}^{n_1}\sum_{j=n_1+1}^{n_1+n_2} h(X_i,X_j,t),
\]
and the two-sample empirical $U$-quantile process
\[
 Q_{n_1,n_2}(p)=U_{n_1,n_2}^{-1}(p)=\inf\{t: U_{n_1,n_2}(t)\geq p\}.
\]
Note that $U_n(\lambda,t)=U_{[n\lambda],n-[n\lambda]}(t)$ and $Q_n(\lambda,p)=Q_{[n\lambda],n-[n\lambda]}(p)$.
Moreover, we define
\[
g(x,y,t)=h(x,y,t)-h_{1}(x,t)-h_{2}(y,t)-U(t),
\]
where $h_1(x,t)$, $h_2(y,t)$, and $U(t)$ have been defined in (\ref{eq:u_t}), (\ref{eq:h_1}) and (\ref{eq:h_2}), respectively. Thus, we obtain the Hoeffding decomposition of the two-sample $U$-statistic as
\[
U_{n_1,n_2}(t)=U(t)+{n_1}^{-1}\sum_{i=1}^{n_1}h_1(X_i,t)+{n_2}^{-1}\sum_{j=n_1+1}^{n_1+n_2}h_2(X_j,t)+({n_1n_2})^{-1}\sum_{i=1}^{n_1}\sum_{j=n_1+1}^{n_1+n_2}g(X_i,X_j,t).
\]
The next two lemmas will deal with the last sum, which is called degenerate part:
\begin{lemma}\label{lem1} Under the assumptions (C1) and (C3), there exists a constant $C$, such that for any integers $0\leq m_1\leq n_1\leq m_2\leq n_2$
\begin{equation}
E\left\{\sum_{i=m_1+1}^{n_1}\sum_{j=m_2+1}^{n_2}g(X_i,X_j,t)\right\}^2\leq C(n_1-m_1)(n_2-m_2),\label{eq:secmoments}
\end{equation}
for all $t$ in the neighborhood referred to in assumption (C3).
\end{lemma}
For the special case $m_2=n_1$, this is Proposition 6.2 of \citet{dehling2012}. The general case can be proved with the same arguments; we omit the details.
\newpage

\begin{lemma}\label{lem2} Suppose that the assumptions (C1) and (C3) hold.
\\[1mm]
(i) There is a constant $C$, such that for all $t$
\[
E\left\{\max_{0\leq m_1\leq n_1\leq m_2\leq n_2\leq 2^l}\Big|\sum_{i=m_1+1}^{n_1}\sum_{j=m_2+1}^{n_2}g(X_i,X_j,t)\Big|\right\}^2\leq C2^{2l}l^4.
\]
(ii) As $n\rightarrow \infty$, we have
\[
\sup_{0\leq \lambda \leq 1}\Bigg|\sum_{i=1}^{\lfloor \lambda n\rfloor}\sum_{j=\lfloor\lambda n\rfloor +1}^{n}g(X_i,X_j,t)\Bigg|=o(n\log^3 n),
\]
almost surely.
\end{lemma}

\begin{proof} To prove the first part of the lemma, we introduce the notation
\[
Q_{m_1,n_1,m_2,n_2}=\sum_{i=m_1+1}^{n_1}\sum_{j=m_2+1}^{n_2}g(X_i,X_j,t),
\]
for $m_1\leq n_1\leq m_2\leq n_2$, and $Q_{m_1,n_1,m_2,n_2}=0$ otherwise. These
quantities satisfy an addition rule
\[
Q_{m_1,n_1,m_2,n_2}+Q_{n_1,n'_1,m_2,n_2}=Q_{m_1,n'_1,m_2,n_2}\\
Q_{m_1,n_1,m_2,n_2}+Q_{m_1,n_1,n_2,n'_2}=Q_{m_1,n_1,m_2,n'_2}
\]
Note that
\[
\max_{0\leq m_1\leq n_1\leq m_2\leq n_2\leq 2^l}|Q_{m_1,n_1,m_2,n_2}|\leq2\max_{0\leq m_1\leq n_1\leq m_2\leq 2^l}|Q_{m_1,n_1,m_2,2^l}|
\leq 4\max_{0\leq n_1\leq m_2\leq 2^l}|Q_{0,n_1,m_2,2^l}|
\]
Now we use a chaining technique. For example,
\[
|Q_{0,5,7,16}|\leq|Q_{0,4,7,8}|+|Q_{0,4,8,16}|+|Q_{4,5,7,8}|+|Q_{4,5,8,16}|.
\]
We conclude that
\[
\max_{0\leq m_1\leq n_1\leq m_2\leq n_2\leq 2^l}|Q_{m_1,n_1,m_2,n_2}|\leq4\, \sum_{d_1=0}^l\sum_{d_2=0}^l\max_{\substack{i=1,\ldots,2^{l-d_1}\\j=1,\ldots,2^{l-d_2}}}|Q_{(i-1)2^{d_1},i2^{d_1},(j-1)2^{d_2},j2^{d_2}}|.
\]
Note that for any random variables $Y_1,\ldots,Y_k$ we have that $E\left(\max_{i=1,\ldots,k}Y_k\right)^2\leq\sum_{i=1}^kEY_i^2$. Using this inequality and (\ref{eq:secmoments}), we conclude that
\begin{eqnarray*}
&& \hspace{-20mm}
E\left(\max_{0\leq m_1\leq n_1\leq m_2\leq n_2\leq 2^l}|Q_{m_1,n_1,m_2,n_2}|\right)^2\\
&&  \leq 16E\Bigg(\sum_{d_1=0}^l\sum_{d_2=0}^l\max_{\substack{i=1,\ldots,2^{l-d_1}\\j=1,\ldots,2^{l-d_2}}}|Q_{(i-1)2^{d_1},i2^{d_1},(j-1)2^{d_2},j2^{d_2}}|\Bigg)^2 \\
&& \leq 16l^2\sum_{d_1=0}^l\sum_{d_2=0}^lE\Bigg(\max_{\substack{i=1,\ldots,2^{l-d_1}\\j=1,\ldots,2^{l-d_2}}}|Q_{(i-1)2^{d_1},i2^{d_1},(j-1)2^{d_2},j2^{d_2}}|\Bigg)^2\\
&& \leq 16l^2\sum_{d_1=0}^l\sum_{d_2=0}^l\sum_{i=1}^{2^{l-d_1}}\sum_{j=1}^{2^{l-d_2}}E\left(Q_{(i-1)2^{d_1},i2^{d_1},(j-1)2^{d_2},j2^{d_2}}\right)^2\\
&& \leq Cl^2\sum_{d_1=0}^l\sum_{d_2=0}^l\sum_{i=1}^{2^{l-d_1}}\sum_{j=1}^{2^{l-d_2}}2^{d_1}2^{d_2}\leq Cl^42^{2l}.
\end{eqnarray*}
So the first part of the lemma is proven. For the second part, it suffices to show that
\[
\max_{0\leq m_1\leq n_1\leq m_2\leq n_2\leq 2^l}|Q_{m_1,n_1,m_2,n_2}|=o(2^{l}l^3).
\]
Now by the Chebyshev inequality, we obtain
\begin{multline*}
\sum_{l=1}^\infty pr\left({2^{-l}l^{-3}}\max_{0\leq m_1\leq n_1\leq m_2\leq n_2\leq 2^l}|Q_{m_1,n_1,m_2,n_2}|\geq\epsilon\right)\\
\leq {\epsilon^{-2}}\sum_{l=1}^\infty {2^{-2l}l^{-6}}E\left(\max_{0\leq m_1\leq n_1\leq m_2\leq n_2\leq 2^l}|Q_{m_1,n_1,m_2,n_2}|\right)^2\leq C\sum_{l=1}^\infty{l^{-2}}<\infty.
\end{multline*}
The Borel-Cantelli lemma completes the proof.
\end{proof}

In order to prove Theorem \ref{th:bahadur}, we need some information about the local behaviour of the empirical $U$-process. We will first concentrate on the first half of the process, that is $\lambda\in[0,1/2]$:

\begin{lemma}\label{lem3}  Under the assumptions (C1), (C2), and  (C3),
\[
\sup_{\substack{\lambda\in[0,1/2]\\|t-t_p|\leq C\left[\frac{\log\log\{\min(\lambda,1-\lambda)\}n}{\min(\lambda,1-\lambda)n}\right]^{1/2}}}\!\!\!\! \lambda(1-\lambda)\left|\left\{U_{\lfloor\lambda n \rfloor,n-\lfloor\lambda n\rfloor}(t)-U(t)\right\}-\left\{U_{\lfloor\lambda n\rfloor,n-\lfloor\lambda n\rfloor}(t_p)-p\right\}\right|
=O\left(n^{-\frac{5}{9}}\right)
\]
almost surely.
\end{lemma}

\begin{proof} We define $n_1=\lfloor n\lambda \rfloor$, $n_2=n-\lfloor n\lambda \rfloor$ and $r_n=({\log\log n_1}/{n_1})^{1/2}$, and note that $n_1+n_2=n$. We define the sequences $c_{2^l}=2^{-5l/9}$, and for $n=2^{l-1}+1,\ldots,2^{l}$ we set $c_n=c_{2^l}$. By the monotonicity of $U_{n_1,n_2}$ and $U$ in $t$, we have that
\begin{eqnarray*}
&&\hspace{-5mm} \sup_{\substack{n_1\leq {n}/{2}\\|t-t_p|\leq Cr_n}}\!\!\!\!\frac{n_1n_2}{n^2}\Big|\{U_{n_1,n_2}(t)-U(t)\}-\{U_{n_1,n_2}(t_p)-p\}\Big|\\[1mm]
&&\leq\!\!\!\!\!\!\!\max_{\substack{n_1\leq {n}/{2}\\t\in c_n\Z\\|t-t_p|\leq Cr_n}} \!\!\!\!\!\!\! \frac{n_1n_2}{n^2}\Big|\{U_{n_1,n_2}(t)-U(t)\}-\{U_{n_1,n_2}(t_p)-p\}\Big|
+\!\!\!\!\!\!\max_{\substack{t\in c_n\Z\\|t-t_p|\leq Cr_n}}|U(t)-U(t+c_n)|.
\end{eqnarray*}
As $U$ is differentiable in $t_p$, we get that the second summand is of the order $O(c_n)$. For the first summand, we use the Hoeffding decomposition and get
\begin{eqnarray}
&&\hspace{-10mm} \max_{\substack{n_1\leq\frac{n}{2}\\t\in c_n\Z\\|t-t_p|\leq Cr_n}} \frac{n_1n_2}{n^2}\Big|\{U_{n_1,n_2}(t)-U(t)\}-\{U_{n_1,n_2}(t_p)-p\}\Big|\label{eq:uproc-osc}\\
&\leq& \max_{\substack{n_1\leq\frac{n}{2}\\t\in c_n\Z\\|t-t_p|\leq Cr_n}}\Big|\frac{1}{n}\sum_{i=1}^{n_1}h_1(X_i,t)-\frac{1}{n}\sum_{i=1}^{n_1}h_1(X_i,t_p)\Big|\nonumber\\
&& +\max_{\substack{n_1\leq\frac{n}{2}\\t\in c_n\Z\\|t-t_p|\leq Cr_n}}\frac{n_1}{n}\Big|\frac{1}{n}\sum_{j=n_1+1}^{n}h_2(X_j,t)-\frac{1}{n}\sum_{j=n_1+1}^{n}h_2(X_j,t_p)\Big|\nonumber\\
&&+\max_{\substack{n_1\leq\frac{n}{2}\\t\in c_n\Z\\|t-t_p|\leq Cr_n}}\Big|\frac{1}{n^2}\sum_{i=1}^{n_1}\sum_{j=n_1+1}^{n_1+n_2}g(X_i,X_j,t)\Big|
+\max_{n_1\leq\frac{n}{2}}\Big|\frac{1}{n^2}\sum_{i=1}^{n_1}\sum_{j=n_1+1}^{n_1+n_2}g(X_i,X_j,t_p)\Big|
\nonumber
\end{eqnarray}
For the first summand, we refer to (13) in Theorem 1 of \citet{wendler2011} and conclude that it is of size $o\{n^{-(5+\gamma)/{8}}(\log n)^{3/4}(\log\log n)^{{1}/{2}}\}=O(n^{-{5}/{9}})$ almost surely for a $\gamma>0$. Note that the continuity condition on the kernel in \cite{wendler2011} is different, but the continuity is only needed to guarantee that $\{h_1(X_i)\}_{i\in\N}$ is near epoch dependent. This also holds under our continuity condition by Proposition 2.11 from \citet{borovkova2001}.

We split the second summand into two parts, so that for the first part $n_1/n$ is small and for the second part $r_n=({{\log\log n_1}/{n_1}})^{1/2}$:
\begin{eqnarray*}
&&\hspace{-10mm}
\max_{\substack{n_1\leq {n}/{2}\\t\in c_n\Z\\|t-t_p|\leq Cr_n}}{n_1}n^{-1}\left|{n}^{-1}\sum_{j=n_1+1}^{n}h_2(X_j,t)-{n}^{-1}\sum_{j=n_1+1}^{n}h_2(X_j,t_p)\right|\\
&\leq& \max_{\substack{n_1\leq n^{{4}/{9}}\\t\in c_n\Z\\|t-t_p|\leq Cr_n}}n^{-5/9}\left|{n}^{-1}\sum_{j=n_1+1}^{n}h_2(X_j,t)-{n}^{-1}\sum_{j=n_1+1}^{n}h_2(X_j,t_p)\right|\\
&& +\max_{\substack{n_1\leq {n}/{2}\\t\in c_n\Z\\|t-t_p|\leq C\left({{\log\log n^{4/9}}/{n^{4/9}}}\right)^{1/2}}}\left|{n}^{-1}\sum_{j=n_1+1}^{n}h_2(X_j,t)-{n}^{-1}\sum_{j=n_1+1}^{n}h_2(X_j,t_p)\right|
\\
&=:&A_1+A_2.
\end{eqnarray*}
As $h$ is bounded and therefore $h_2$ is bounded, we have that $A_1=O(n^{-{5}/{9}})$. Along the lines of the proof of  Theorem 1 in \cite{wendler2011}, we obtain
\[
A_2=O\{n^{-{2}({1+\gamma})/{36}}n^{-{1}/{2}}(\log n)^{3/4}(\log\log n)^{{1/2}}\}=O(n^{-{5}/{9}})
\] 
almost surely. For the third summand on the r.h.s. of (\ref{eq:uproc-osc}), we use the first part of Lemma A1, the Chebyshev inequality, and the fact that the second moment of the maximum of random variables is smaller or equal to the sum of second moments. We obtain
\begin{eqnarray*}
&& \sum_{l=1}^\infty pr\Bigg\{{c_{2^l}^{-1}}\max_{2^{l-1}\leq n\leq 2^l}\max_{\substack{n_1\leq {n}/{2}\\t\in c_n\Z\\|t-t_p|\leq Cr_n}}\left|{n^{-2}}\sum_{i=1}^{n_1}\sum_{j=n_1+1}^{n_1+n_2}g(X_i,X_j,t )\right|\geq\epsilon\Bigg\}\\
&&\leq\sum_{l=1}^\infty{c^{-2}_{2^l}\epsilon^{-2} 2^{-4(l-1)}}
E\Bigg\{\max_{2^{l-1}\leq n\leq 2^l}\max_{\substack{n_1\leq {n}/{2}\\t\in c_n\Z\\|t-t_p|\leq Cr_n}}\left|\sum_{i=1}^{n_1}\sum_{j=n_1+1}^{n_1+n_2}g(X_i,X_j,t)\right|\Bigg\}^2\\
&&\leq\sum_{l=1}^\infty\sum_{\substack{t\in c_{2^l}\Z\\|t-t_p|\leq C}}{c^{-2}_{2^l}\epsilon^{-2} 2^{-4(l-1)}} E\left\{\max_{0\leq n_1\leq n_1+ n_2\leq 2^{l}}\left|\sum_{i=1}^{n_1}\sum_{j=n_1+1}^{n_1+n_2}g(X_i,X_j,t)\right|\right\}^2\\
&&\leq 2^5C\sum_{l=1}^\infty {c^{-3}_{2^{l}}2^{-4l}}2^{2l}l^4\leq 2^5C\sum_{l=1}^\infty {l^4}{2^{-{l}/{3}}}<\infty,
\end{eqnarray*}
as the set $\{t\in c_{2^l}\Z,\ |t-t_p|\leq C\}$ has at most $2C\, c_{2^l}^{-1}$ elements. Using the Borel-Cantelli lemma, we conclude that the third summand on the r.h.s. of (\ref{eq:uproc-osc}) is of size $o(c_n)$ almost surely. The last summand can be treated in the same way and so in total we have proved the order $O(n^{-{5}/{9}})$ almost surely.
\end{proof}

\begin{lemma}\label{lem4} Under the assumptions (C1) and (C3),
\[
\sup_{n_2>n_1}\left|U_{n_1,n_2}(t_p)-p\right|=O\left\{\left({{\log\log n_1}/{n_1}}\right)^{1/2}\right\}
\]
almost surely.
\end{lemma}

\begin{proof} We use the Hoeffding decomposition
\begin{multline*}
\sup_{n_2>n_1}\left|U_{n_1,n_2}(t_p)-p\right|\leq \left|{n_1}^{-1}\sum_{i=1}^{n_1}h_1(X_i,t_p)\right|+\left|{n_2}^{-1}\sum_{i=1}^{n_1}h_2(X_i,t_p)\right|\\
+\sup_{n_2>n_1}{n_2}^{-1}\left|\sum_{j=1}^{n_1+n_2}h_2(X_j,t_p)\right|+\sup_{n_2>n_1}({n_1n_2})^{-1}\left|\sum_{i=1}^{n_1}\sum_{j=n_1+1}^{n_2}g(X_i,X_j,t_p)\right|.
\end{multline*}
For the first two summands, we use Proposition 3.7 of \cite{wendler2011}, which leads to
\[
\left|{n_1}^{-1}\sum_{i=1}^{n_1}h_k(X_i,t_p)\right|=O\left\{\left({{\log\log n_1}/{n_1}}\right)^{1/2}\right\}
\]
for $k=1,2$ almost surely. Furthermore
\[
\sup_{n_2>n_1}{n_2}^{-1}\left|\sum_{j=1}^{n_1+n_2}h_2(X_j,t_p)\right|\leq \sup_{n_2>n_1}2({n_1+n_2})^{-1}\left|\sum_{j=1}^{n_1+n_2}h_2(X_j,t_p)\right|=O\left\{\left({{\log\log n_1}/{n_1}}\right)^{1/2}\right\}.
\]
For the last summand, we use Lemma A1 to obtain
\[
E\left(\max_{\substack{0\leq m_1\leq n_1\leq 2^{l_1}\\ n_1\leq m_2\leq n_2\leq 2^{l_2}}}|Q_{m_1,n_1,m_2,n_2}|\right)^2\leq Cl_1^2l_2^22^{l_1}2^{l_2}.
\]
Now by the Chebyshev inequality, we obtain
\begin{multline*}
\sum_{l_2=1}^\infty\sum_{l_1=1}^{l_2} pr\left\{\left({2^{l_1}\log l_1}\right)^{-1/2}{2^{-l_2}}\max_{\substack{0\leq m_1\leq n_1\leq 2^{l_1}\\ n_1\leq m_2\leq n_2\leq 2^{l_2}}}|Q_{m_1,n_1,m_2,n_2}|\geq\epsilon\right\}\\
\leq \epsilon^{-2}\sum_{l_2=1}^\infty\sum_{l_1=1}^{l_2}\left({2^{l_1}\log l_12^{2l_2}}\right)^{-1/2}E\left(\max_{\substack{0\leq m_1\leq n_1\leq 2^{l_1}\\ n_1\leq m_2\leq n_2\leq 2^{l_2}}}|Q_{m_1,n_1,m_2,n_2}|\right)^2\\
\leq C\sum_{l_2=1}^\infty\sum_{l_1=1}^{l_2}{l_1^2l_2^2}/({\log l_12^{l_2}})^{-1}<\infty,
\end{multline*}
so we can conclude that the last summand is of the required  order almost surely.\end{proof}

\begin{proposition}\label{prop1} Under the assumptions (C1), (C2) and (C3), the process
\[
{n}^{1/2}\left\{\lambda(1-\lambda)(U_{\lfloor \lambda n \rfloor,n- \lfloor \lambda n \rfloor}(t_p)-p\right\}_{\lambda\in[0,1]}
\]
converges weakly to
\[
\left[(1-\lambda) W_1(\lambda)+\lambda\{W_2(1)-W_2(\lambda)\}\right]_{\lambda\in[0,1]},
\]
where $W=(W_1,W_2)$ is a two-dimensional Brownian motion with covariance structure
\[
 \Cov\{W_i(\mu),W_j(\lambda)\}=(\mu\wedge \lambda)
 \sum_{k\in \Z} E\left[h_{i}\{X_0; Q(p)\}, h_{j}\{X_k;Q(p)\}\right].
\]
\end{proposition}

This is Theorem 2.4 of \citet{dehling2015}.

\begin{lemma}\label{lem5} Under the assumptions (C1), (C2) and (C3) for any bandwidth $b=b_n$ with $b+b^{-3}n^{-1}=o(1)$
we have the following convergence in probability:
\[
\hat{u}(0)=\frac{2}{n(n-1)b}\sum_{1\leq i<j\leq n}K\left(\frac{X_i-X_j}{b}\right)\rightarrow u(0)
\]
\end{lemma}

\begin{proof} First note that $\hat{u}$ is a one-sample $U$-statistic with symmetric kernel $k_n(x,y)={b}^{-1}K\{({x-y})/{b}\}$ depending on $n$. We use the Hoeffding decomposition
\begin{eqnarray*}
\tilde{u}_n&=&Ek_n(X,Y),\\
k_{1,n}(x)&=&Ek_n(x,X)-\tilde{u}_n,\\
k_{2,n}(x,y)&=&k_n(x,y)-k_{1,n}(x)-k_{1,n}(y)-\tilde{u}_n,
\end{eqnarray*}
where $X$, $Y$ are independent with the same distribution as $X_1$. We obtain
\[
\hat{u}(0 )=\tilde{u}_n+{2}{n}^{-1}\sum_{i=1}^nk_{1,n}(X_i)+2\{{n(n-1)}\}^{-1}\sum_{1\leq i<j\leq n}k_{2,n}(X_i,Y_i).
\]
By our assumptions, $K$ has a bounded support, so let $K(x)=0$ for $|x|>M$. Because the density $u$ is continuous and $K$ integrates to 1, we can conclude that for $n\to\infty$ we have
\[
|\tilde{u}_n-u(0)|=|\int {b}^{-1}K({x}/{b})u(x)dx-u(0)|\leq\sup_{|x|\leq Mb}\left|u(x)-u(0)\right|\rightarrow 0,
\]
since $b_n\rightarrow 0$. As $K$ is Lipschitz continuous, $|K(x)-K(y)|\leq L_1|x-y|$ for some constant $L_1$, we have that $k_{1,n}(x)$ is Lipschitz continuous with constant $L_1/b^2$. By Proposition 2.11 of \cite{borovkova2001}, it follows that $\left\{k_{1,n}(X_i)\right\}_{i\in\N}$ is near epoch dependent with approximation constants $a'_k=3(a_k)^{1/2}/b_n^2$. Let $C_1=C/b$ be the upper bound of $k_{1,n}(X_i)$, then by Lemma 2.18 of \citet{borovkova2001}
\[
\left|E\left\{k_{1,n}(X_i)k_{1,n}(X_{j})\right\}\right|\leq 4C_1a'_{|i-j|/3}+2C^2_1\beta(|i-j|/3)\leq C{b^{-3}}\{{a_{|i-j|/3}^{1/2}}+\beta(|i-j|/3)\},
\]
so we obtain by stationarity that
\begin{eqnarray*}
E\Big\{{2}{n}^{-1}\sum_{i=1}^nk_{1,n}(X_i)\Big\}^2
&\leq&{4}{n}^{-1}\sum_{i=1}^\infty\left|E\left\{ k_{1,n}(X_1)k_{1,n}(X_i)\right\}\right|\\
&\leq& C{n^{-1}b^{-3}}\sum_{i=1}^\infty\{({a_{|i-j|/3}})^{1/2}+\beta(|i-j|/3)\}\rightarrow 0,
\end{eqnarray*}
because $nb^3\rightarrow\infty$, so the second summand converges to 0. For the third summand, we use Lemma 4.3 of \citet{borovkova2001} and the fact that $k_{2,n}(x,y)$ is a degenerate kernel bounded by $4C_1/b$ and that the product $k_{2,n}(x_1,x_2)k_{2,n}(x_3,x_4)$ is $1$-Lipschitz with constant $4\{4({C_1}/{b})({L_1}/{b^2})\}=Cb^{-3}$. We get the inequality
\[
\left|E\left\{k_{2,n}(X_{i_1},X_{i_2})k_{2,n}(X_{i_3},X_{i_4})\right\}\right|\leq {C}\left\{A_{m/3}+\beta(m/3)\right\}/b^{2}+CA_{m/3}/b^{3}
\]
with $A_i=\left({2\sum_{n=i}^\infty a_n}\right)^{1/2}$ and $m=\max\left\{i_{\left(2\right)}-i_{\left(1\right)},i_{\left(4\right)}-i_{\left(3\right)}\right\}$, where   $i_{(1)}\leq i_{(2)}\leq i_{(3)}\leq i_{(4)}$
are the order statistics of the indices $i_1,i_2,i_3,i_4$. Thus, we obtain
\begin{multline*}
E\Big\{{2}{n^{-1}(n-1)^{-1}}\sum_{1\leq i<j\leq n}k_{2,n}(X_i,Y_i)\Big\}^2\leq
C{n^{-4}}\sum_{i_{1}<i_{2},i_{3}<i_{4}}^{n}\left|E\left\{k_{2,n}(X_{i_1},X_{i_2})k_{2,n}(X_{i_3},X_{i_4})\right\}\right|\\
=C{n^{-4}}\sum_{m=0}^{n}\sum_{\substack{i_{1}<i_{2},i_{3}<i_{4}\\ \max\{i_{(2)}-i_{(1)},i_{(4)}-i_{(3)}\}=m}}\left|E\left\{k_{2,n}(X_{i_1},X_{i_2})k_{2,n}(X_{i_3},X_{i_4})\right\}\right|\\
\leq C{n^{-4}b^{-3}}\sum_{m=0}^{n}\sum_{\substack{i_{1}<i_{2},i_{3}<i_{4}\\ \max\{i_{(2)}-i_{(1)},i_{(4)}-i_{(3)}\}=m}}\Big\{A_{{m}/{3}}+\beta({m}/{3})\Big\}.
\end{multline*}
At this point, we have to calculate the number of quadruples $(i_{1},i_{2},i_{3},i_{4})$ such that $\max\{i_{(2)}-i_{(1)},i_{(4)}-i_{(3)}\}=m$. First note that there are at most 6 quadruples which lead to the same ordered numbers $i_{(1)},i_{(2)}, i_{(3)},i_{(4)}$. There are at most $n^2$ possibilities to choose $i_{(1)}$ and $i_{(4)}$. If $i_{(2)}-i_{(1)}=\max\{i_{(2)}-i_{(1)},i_{(4)}-i_{(3)}\}=m$, then $i_{(2)}$ is already fixed and there are $m$ possibilities for $i_{(3)}$. The same argument applies if $i_{(4)}-i_{(3)}=\max\{i_{(2)}-i_{(1)},i_{(4)}-i_{(3)}\}=m$, so we finally obtain
\[
E\Big\{{2}{n^{-1}(n-1)^{-1}}\sum_{1\leq i<j\leq n}k_{2,n}(X_i,Y_i)\Big\}^2
\leq C{n^{-2}b^{-3}}\sum_{m=0}^{n}m\Big\{A_{{m}/{3}}+\beta({m}/{3})\Big\}\rightarrow 0,
\]
as the $mA_{{m}/{3}}$ and $m\beta({m}/{3})$ are summable by assumption (C1), and $n^2b^3\rightarrow\infty$.
\end{proof}

\begin{lemma}\label{lemma:unifu0}
If $nb^4\rightarrow \infty$, then we have the following convergence in probability under $\Hy_0$:
\[
 \max_k|\hat{u}_{k,n}(0) - \hat{u}(0)| \rightarrow 0
\]
\end{lemma}

\begin{proof}
Define the estimate $\hat\Delta_k=\med\{(X_j-X_i):1\le i\le k<j\le n\}$ of the height of a possible level shift at time $k$ and the corrected data
\[
X_i^{(k)}:=\left\{
 \begin{array}{ll}
  X_i, & i\leq k \\[2mm]
  X_i-\hat\Delta_k, & i \geq k+1.
 \end{array}
\right.
\]
The density estimator based on the corrected data is then given by
\[
  \hat{u}_{k,n}(0)= \frac{2}{n(n-1) b} \sum_{1\leq i<j\leq n} K\left( \frac{X_i^{(k)}-X_j^{(k)}}{b}  \right).
\]
We analyze this density estimator by comparing it with the density estimator based on the original data, given by
\[
  \hat{u}(0)= \frac{2}{n(n-1) b} \sum_{1\leq i<j\leq n} K\left( \frac{X_i-X_j}{b}  \right).
\]

Observe that  for $1\leq i<j\leq n$, we have $X_i^{(k)}-X_j^{(k)}=X_i-X_j$, unless $1\leq i \leq k< j\leq n$. Thus we obtain
\[
  \hat{u}_{k,n}(0) -\hat{u}(0) =\frac{2}{n(n-1)b} \sum_{i=1}^k
  \sum_{j=k+1}^n  \left\{ K\left( \frac{X_i^{(k)}-X_j^{(k)}}{b}  \right) - K\left( \frac{X_i-X_j}{b}  \right)   \right\}.
\]
Denoting the Lipschitz constant of $K$ by $L$, we get
\[
  \max_{k} |\hat{u}_{k,n}(0) -\hat{u}(0)| \leq \max_k\frac{2(n-k)k }{n(n-1)b} L
    \frac{\left|\hat\Delta_k\right|}{b}  ={{n}^{-1/2}b^{-2}}\max_k\frac{2L{n}^{1/2} k(n-k)\left|\hat\Delta_k\right|}{n(n-1)}\to 0,
\]
if ${n}^{1/2}b^{2}\to\infty$, since the second ratio converges to the supremum of a Brownian bridge according to Theorem \ref{th:HL-process}.
\end{proof}

\begin{lemma}\label{lem6} Let $G$ be a non-decreasing function, $c,l>0$ constants and $[C_1,C_2]\subset\R$. If for all $t,t'\in[C_1,C_2]$ with $|t-t'|\leq l+2c$
\[
|G(t)-G(t')-(t-t')|\leq c,
\]
then for all $p,p'\in\R$ with $|p-p'|\leq l$ and $G^{-1}(p),G^{-1}(p')\in(C_1+2c+l,C_2-2c-l)$
\[
|G^{-1}(p)-G^{-1}(p')-(p-p')|\leq c
\]
where $G^{-1}(p)=\inf\left\{t\big|G(t)\geq p\right\}$ denotes the generalized inverse.
\end{lemma}

\begin{proof}
This is Lemma 3.5 of \cite{wendler2012}.
\end{proof}
\newpage

\section*{Appendix 2: proofs of the main theorems}

{\em Proof of Theorem \ref{th:bahadur}.} Without loss of generality, we can assume that $u(t_p)=1$, otherwise replacing $h(x,y,t)$ by $h\{x,y,{t}/{u(t_p)}\}$. We will first concentrate on the first half, that means we will investigate
\begin{multline*}
\sup_{\lambda\in[0,\frac{1}{2}]}\lambda(1-\lambda)\left|U^{-1}_{\lfloor \lambda n \rfloor,n-\lfloor\lambda n \rfloor}(p)-t_p+U_{\lfloor \lambda n \rfloor,n-\lfloor\lambda n \rfloor}(t_p)-p\right|\\
\leq \sup_{\lambda\in[0,\frac{1}{2}]}\lambda(1-\lambda)\left|U^{-1}_{\lfloor \lambda n \rfloor,n-\lfloor\lambda n \rfloor}(p)-U^{-1}_{\lfloor \lambda n \rfloor,n-\lfloor\lambda n \rfloor}\left\{U_{\lfloor \lambda n \rfloor,n-\lfloor\lambda n \rfloor}(t_p)\right\}+U_{\lfloor \lambda n \rfloor,n-\lfloor\lambda n \rfloor}(t_p)-p\right|\\
+\sup_{\lambda\in[0,\frac{1}{2}]}\lambda(1-\lambda)\left|U^{-1}_{\lfloor \lambda n \rfloor,n-\lfloor\lambda n \rfloor}\left\{U_{\lfloor \lambda n \rfloor,n-\lfloor\lambda n \rfloor}(t_p)\right\}-t_p\right|.
\end{multline*}
Define $r_n= \left\{{{\log\log (\lambda n)/(\lambda n)}}\right\}^{1/2}$. By Lemma A\ref{lem4}, we can choose $C_1>0$, such that for all n
\[
pr\left[\sup_{\lambda\in[0,\frac{1}{2}]}\left|U_{\lfloor \lambda n \rfloor,n-\lfloor\lambda n \rfloor}(t_p)-p\right|/r_n\geq C_1 \right]\leq \epsilon.
\]
Hence, using Lemma A\ref{lem3} and A\ref{lem6}, there exists a constant $C_2$ such that
\begin{eqnarray*}
&& \hspace{-8mm}pr\Big[\! \sup_{\lambda\in[0,\frac{1}{2}]}\! \lambda(1-\lambda)
 \big|U^{-1}_{\lfloor \lambda n \rfloor,n-\lfloor\lambda n \rfloor}(p)-U^{-1}_{\lfloor \lambda n \rfloor,n-\lfloor\lambda n \rfloor}\{U_{\lfloor \lambda n \rfloor,n-\lfloor\lambda n \rfloor}(t_p)\}+U_{\lfloor \lambda n \rfloor,n-\lfloor\lambda n \rfloor}(t_p)-p\big|>C_2n^{-5/9}\Big]\\
&& \leq pr\Bigg[\sup_{\substack{\lambda\in[0,\frac{1}{2}]\\ |p-p'|\leq C_1 r_n}}\lambda(1-\lambda)\left|U^{-1}_{\lfloor \lambda n \rfloor,n-\lfloor\lambda n \rfloor}(p)-U^{-1}_{\lfloor \lambda n \rfloor,n-\lfloor\lambda n \rfloor}(p')+p'-p\right|>C_2n^{-5/9}\Bigg]\\
&&\qquad +pr\Bigg[\sup_{\lambda\in[0,\frac{1}{2}]}\left|U_{\lfloor \lambda n \rfloor,n-\lfloor\lambda n \rfloor}(t_p)-p\right|/r_n\geq C_1 \Bigg]\\
&& \leq pr\Bigg(\!\!\sup_{\substack{\lambda\in[0,\frac{1}{2}]\\ |t-t_p|\leq C_1 r_n}}\lambda(1-\lambda)\left|U_{\lfloor \lambda n \rfloor,n-\lfloor\lambda n \rfloor}(t)-U(t)-U_{\lfloor \lambda n \rfloor,n-\lfloor\lambda n \rfloor}(t_p)+p\right|>C_2n^{-5/9}\Bigg)+\epsilon\\
&& \leq 2\epsilon.
\end{eqnarray*}
Thus, the first summand is of order $n^{-5/9}$. It remains to show the convergence of the second summand $U_{\lfloor \lambda n \rfloor,n-\lfloor\lambda n \rfloor}^{-1}\{U_{\lfloor \lambda n \rfloor,n-\lfloor\lambda n \rfloor}(t_p)\}-t_p$. By the definition of the generalized inverse,
$$U_{\lfloor \lambda n \rfloor,n-\lfloor\lambda n \rfloor}^{-1}\{U_{\lfloor \lambda n \rfloor,n-\lfloor\lambda n \rfloor}(t_p)\}-t_p\leq 0.$$ Furthermore, if $U_{n_1,n_2}(t)<U_{n_1,n_2}(t_p)$ by the monotonicity of $h$, we have for all $n_2'\geq n_2$ that $U_{n_1,n'_2}(t)<U_{n_1,n'_2}(t_p)$. As $U_{n_1,n_2}^{-1}\{U_{n_1,n_2}(t_p)\}$ is the supremum of all $t$ such that $U_{n_1,n_2}(t)<U_{n_1,n_2}(t_p)$, it follows that $U_{n_1,n_2}^{-1}\{
U_{n_1,n_2}(t_p)\}$ is nondecreasing in $n_2$.

For every $c>0$, the inequality $U_{n_1,n_1}^{-1}\{U_{n_1,n_1}(t_p)\}-t_p<-c$ implies that $U_{n_1,n_1}(t_p-c)- U_{n_1,n_1}(t_p)\geq0$, which is equivalent to
\begin{equation*}
U_{n_1,n_1}(t_p-c)- U_{n_1,n_1}(t_p)-U(t_p-c)+p\geq- U(t_p-c)+p.
\end{equation*}
By Lemma A\ref{lem3}, there is a constant $C_3$ such that
\begin{equation*}
pr\Bigg(\sup_{n_1\in\N}n_1^{{5}/{9}}\sup_{|t-t_p|\leq r_n}\left|U_{n_1,n_1}(t)-U(t)-\{U_{n_1,n_1}(t_p)-p\}\right|> C_3\Bigg)<\epsilon.
\end{equation*}
As $U$ is differentiable, we have that $U(t_p-C_4n_1^{-{5}/{9}})+p>C_3n_1^{-{5}/{9}}$ for some constant $C_4$ and consequently for all $n_2\geq n_1$
\begin{equation*}
U_{n_1,n_2}^{-1}\{U_{n_1,n_2}(t_p)\}-t_p\geq -C_4n_1^{-{5}/{9}}.
\end{equation*}
Finally we have that $\lambda(1-\lambda)\lfloor\lambda n\rfloor^{-{5}/{9}}\leq n^{-{5}/{9}}$, and so we arrive at
\begin{eqnarray*}
&& pr\Bigg[\sup_{\lambda\in[0,\frac{1}{2}]}\lambda(1-\lambda)\left|U^{-1}_{\lfloor \lambda n \rfloor,n-\lfloor\lambda n \rfloor}\{U_{\lfloor \lambda n \rfloor,n-\lfloor\lambda n \rfloor}(t_p)\}-t_p\right|>C_4n^{-5/9}\Bigg]\\
&&\leq pr\left[\sup_{n_1\leq n/2}n_1^{5/9}\left|U^{-1}_{n_1,n_1}\{U_{n_1,n_1}(t_p)\}-t_p\right|>C_4\right]\\
&&\leq pr\Bigg[\sup_{n_1\in\N}n_1^{5/9}\left|U_{n_1,n_1}(t_p-C_4n_1^{-5/9})-U(t_p-C_4n_1^{-5/9})-\{U_{n_1,n_1}(t_p)-p\}\right|>C_3\Bigg)\\
&&\leq pr\Bigg[\sup_{n_1\in\N}n_1^{{5}/{9}}\sup_{|t-t_p|\leq r_n}\left|U_{n_1,n_1}(t)-U(t)-\{U_{n_1,n_1}(t_p)-p\}\right|> C_3\Bigg]<\epsilon,\end{eqnarray*}
and we have shown the convergence in probability for $\lambda$ restricted to $[0,{1}/{2}]$. For $\lambda\in [1/2,1]$, note that
\begin{multline*}
\sup_{\lambda\in[\frac{1}{2},1]}\lambda(1-\lambda)\left|U^{-1}_{\lfloor \lambda n \rfloor,n-\lfloor\lambda n \rfloor}(p)-t_p+U_{\lfloor \lambda n \rfloor,n-\lfloor\lambda n \rfloor}(t_p)-p\right|\\
=\sup_{\lambda\in[0,\frac{1}{2}]}\lambda(1-\lambda)\left|\tilde{U}^{-1}_{\lfloor \lambda n \rfloor,n-\lfloor\lambda n \rfloor}(p)-t_p+\tilde{U}_{\lfloor \lambda n \rfloor,n-\lfloor\lambda n \rfloor}(t_p)-p\right|,
\end{multline*}
where $\tilde{U}_{n_1,n_2}$ is the two sample $U$-statistics with kernel $\tilde{h}(x,y,t)=h(y,x,t)$ calculated for the stochastic process $(\tilde{X}_i)_{i\in\Z}$ with $\tilde{X}_{i}=X_{n-i}$. Because of stationarity, the probability distribution of this does not change if we insert the random variables $\tilde{X}'_{i}=X_{-i}$ instead. The process $(X_{-i})_{i\in\Z}$ inherits the near epoch properties of $(X_i)_{i\in\Z}$.  Thus, with  the same arguments as above, we obtain
\begin{equation*}
\sup_{\lambda\in[0,\frac{1}{2}]}\lambda(1-\lambda)\left|\tilde{U}^{-1}_{\lfloor \lambda n \rfloor,n-\lfloor\lambda n \rfloor}(p)-t_p+\tilde{U}_{\lfloor \lambda n \rfloor,n-\lfloor\lambda n \rfloor}(t_p)-p\right|=O_P(n^{-5/9}).\qed
\end{equation*}
\vspace*{3mm}

{\em Proof of Theorem \ref{th:2uq-fclt}.}
We decompose the stochastic process into two parts:
\begin{multline*}
{n}^{1/2}\left[\lambda(1-\lambda)\{U^{-1}_{\lfloor \lambda n \rfloor,n-\lfloor \lambda  n \rfloor}(p)-t_p\}\right]_{\lambda\in[0,1]}\\
={n}^{1/2}\left[\lambda(1-\lambda)\{p-U_{\lfloor \lambda n \rfloor,n-\lfloor \lambda n \rfloor}(t_p)\}/{u(t_p)}\right]_{\lambda\in[0,1]}\\
+{n}^{1/2}\left[\lambda(1-\lambda)\left\{U^{-1}_{\lfloor \lambda n \rfloor,n-\lfloor  \lambda n \rfloor}(p)-t_p+\frac{U_{\lfloor \lambda n \rfloor,n-\lfloor \lambda n \rfloor}(t_p)-p}{u(t_p)}\right\}\right]_{\lambda\in[0,1]}.
\end{multline*}
By Theorem \ref{th:bahadur}, the second part converges to zero in supremum norm. As a consequence of Proposition A\ref{prop1}, the first part converges weakly to
\begin{equation*}
\left[(1-\lambda) W_1(\lambda)+\lambda\{W_2(1)-W_2(\lambda)\}\right]_{\lambda\in[0,1]},
\end{equation*}
where $W=(W_1,W_2)$ is a two-dimensional Brownian motion with covariance structure
\[
 \Cov\{W_i(\mu),W_j(\lambda)\}=(\mu\wedge \lambda){u^{-2}\{Q(p)\}}
 \sum_{k\in \Z} E[h_{i}\{X_0; Q(p)\}, h_{j}\{X_k;Q(p)\}].\qed
\]

\newpage

{\em Proof of Theorem \ref{th:HL-process} and Corollary \ref{cor:HL-test}.} By Theorem~\ref{th:2uq-fclt},
\begin{equation*}
{n}^{1/2}\lambda(1-\lambda)\{{\med}\left(X_i-X_j\big|1\leq i\leq \lfloor n\lambda\rfloor,\lfloor n\lambda\rfloor+1\leq j\leq n\right)\}_{\lambda\in(0,1)}
\end{equation*}
converges to
\begin{equation*}
\left[(1-\lambda) W_1(\lambda)+\lambda\{W_2(1)-W_2(\lambda)\}\right]_{\lambda\in[0,1]},
\end{equation*}
where $W=(W_1,-W_1)$ and $W_1$ is a Brownian motion, as $h_1(x,0)=-h_2(x,0)$. The variance is $\var \{W_1(1)\}={\sigma^2}/{u^2(0)}$. Now
\begin{equation*}
\frac{u(0)}{\sigma}\left[(1-\lambda) W_1(\lambda)+\lambda\{-W_1(1)+W_1(\lambda)\}\right]=\frac{u(0)}{\sigma}W_1(\lambda)-\lambda \frac{u(0)}{\sigma}W_1(1)
\end{equation*}
is a Brownian bridge. Finally, by Lemma A\ref{lem5}, Lemma A\ref{lemma:unifu0} and Theorem 1.2 of \cite{dehling2013}, ${\hat{u}_{k,n}}/{\hat{\sigma}_n}\rightarrow{u(0)}/{\sigma}$ in probability, which completes the proof also of Corollary \ref{cor:HL-test}. \qed

{\em Proof of Theorem \ref{theo:consistent}.}
For simplicity of notation, we write $k^\ast=k_n^\ast=[n\tau]$. Observe that
\begin{eqnarray*}
T_n &\geq &  \frac{\hat{u}_{k_n^\ast,n}(0)}{\hat{\sigma}_n} \frac{k_n^\ast}{n} \left(1-\frac{k_n^\ast}{n}\right){n}^{1/2}
\left| \med\{(X_j-X_i):1\leq i \leq k_n^\ast, k_n^\ast+1\leq j \leq n   \}\right|\\
&=:&\frac{\hat{u}_{k_n^\ast,n}(0)}{\hat{\sigma}_n}M_{n,k_n^\ast}
\end{eqnarray*}
Now, we can apply Theorem 2.4 of Dehling and Fried (2012), and we obtain
\[
  {n}^{1/2}\left[ \med\{(X_j-X_i): 1\leq i\leq k_n^\ast, k_n^\ast+1 \leq j\leq n  \} -\Delta \right] \claw \frac{1}{U^\prime(\Delta)} W_\Delta.
\]
Thus $M_{n,k_n^\ast}\geq \tau(1-\tau){n}^{1/2}\Delta+O_P(1)$. Note that the corrected estimate $\hat{u}_{k_n^\star,n}(0)$, which is used in $T_n$ at the true change-point position $k_n^\star$, is consistent for $u(0)$ not only under $\Hy_0$ but also under the alternative,
since the effect of $\Delta$ cancels out due to correcting with an equivariant estimator.
Thus showing $\hat{\sigma}_n{n}^{-1/2}\rightarrow 0$ as $n\rightarrow\infty$ is sufficient  for the consistency of the test based on $T_n$, since the above lower bound for $T_n$, and also $T_n$ itself, will converge to infinity in probability if $\Delta\neq 0$.
Recall that
\begin{equation*}
\hat{\sigma}_n = \frac{{\pi}^{1/2}}{ ({2l_n})^{1/2}(n-l_n+1)} \sum_{i=0}^{n-l_n} \bigg|\sum_{j=i+1}^{i+l_n} \big\{F_n(X_i)-1/2\big\}\bigg|.
\end{equation*}
Because $F_n$ is bounded by $0$ and $1$, we have $|\sum_{j=i+1}^{i+l_n} \{F_n(X_i)-1/2\}|\leq l_n/2$.
 Since we have assumed that $l_n=o(n^{1/2})$, it follows that $\hat{\sigma}_n\leq {({\pi l_n}/8)^{1/2}}=o(n^{1/4})$, which completes the proof.

In fact,  the test is even consistent if the corrected estimates $\hat{u}_{k,n}(0)$ in the definition of $T_n$ are replaced by the
simpler uncorrected estimate $\hat{u}(0)$. For this we just need to argue that additionally $\hat{u}(0)\geq {v}_n\rightarrow v$ in probability as $n\rightarrow\infty$, for some random variables $v_n$ and some constant $v>0$.  Observe that
\begin{multline*}
\hat{u}(0) = \frac{2}{n(n-1)b} \sum_{1\leq i<j\leq n} K\left(\frac{X_i-X_j}{b}\right)\\
\geq \frac{2}{n(n-1)b} \sum_{1\leq i<j\leq k_n^\star} K\left(\frac{X_i-X_j}{b}\right)+\frac{2}{n(n-1)b} \sum_{k_n^\star< i<j\leq n} K\left(\frac{X_i-X_j}{b}\right):=v_n.
\end{multline*}
Because the sample $X_i$, $i\leq k_n^\star$ is not affected by the change and the same is true for $X_i$, $k_n^\star<i\leq n$, we can apply Lemma A\ref{lem5}  to obtain
\begin{multline*}
v_n=\frac{k_n^\star(k_n^\star-1)}{n(n-1)}\frac{2}{k_n^\star(k_n^\star-1)b} \sum_{1\leq i<j\leq k_n^\star} K\left(\frac{X_i-X_j}{b}\right)\\
+\frac{(n-k_n^\star)(n-k_n^\star-1)}{n(n-1)}\frac{2}{(n-k_n^\star)(n-k_n^\star-1)b} \sum_{k_n^\star< i<j\leq n} K\left(\frac{X_i-X_j}{b}\right)\xrightarrow{n\rightarrow\infty}\{\tau^2+(1-\tau)^2\}u(0)
\end{multline*}
in probability, which is larger than 0 by assumption. \qed

\newpage

\section*{Appendix 3: further simulation results}

In this appendix we report some further simulation results.

In some additional simulation experiments we have investigated the estimation of the long run variance $\sigma^2$ under the alternative of a single level shift. For this we have generated time series of length $N=200$, $N=500$ or $N=1000$ from a first order autoregressive model with lag one correlation $\phi_1=$0$\cddot$4 and Gaussian innovations with unit variance, including a level shift of increasing height $j/4$, $j=0,1,\ldots,9$, in the center of the data set, that is, after the time point $N/2$. 600 data sets have been generated for each scenario.
Figure \ref{fig:asymvarest} depicts simulation results for estimation of the square root of the long run variance of the cumulative sums or the Hodges-Lehmann change-point statistic. The adaptive subsampling estimators tend to underestimate the long run standard deviation, which is depicted by a horizontal line, under the hypothesis. This negative bias gets smaller with increasing sample size, since the chosen block lengths get larger then.
Note that the true standard deviation of the change-point statistics in finite samples is somewhat smaller than this asymptotical limit, so the underestimation under the hypothesis is not severe if $N$ is large. Splitting the time series into non-overlapping subsequences of equal size stabilizes the estimation of the standard deviation under the alternative of a change-point but increases the underestimation under the hypothesis if $N$ is not very large, resulting in an increased probability of a type one error. Nevertheless, in case of long time series such splitting strategies will lead to a substantial further increase of the power against shifts.

\begin{figure}[htbp!]\centering
\resizebox{14.5cm}{5.0cm}{\includegraphics{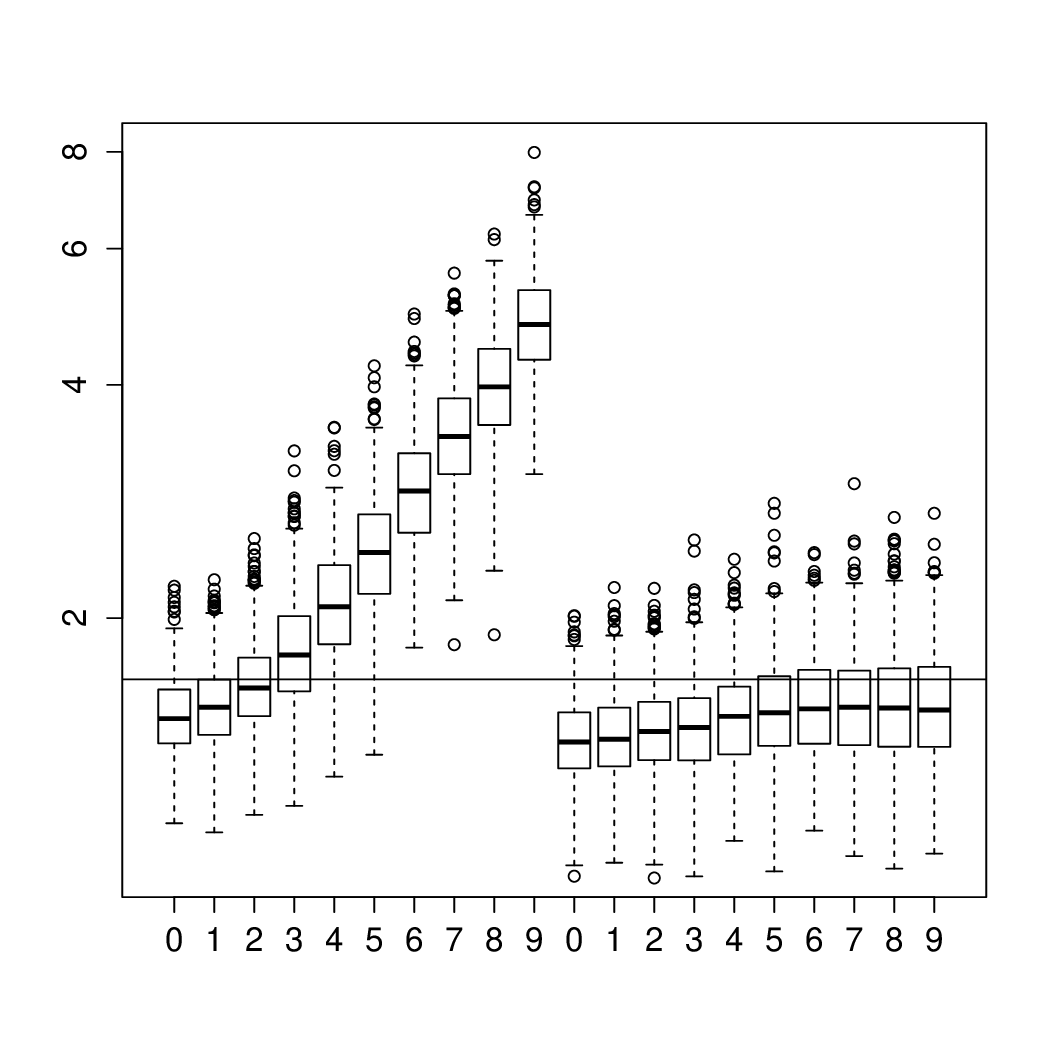}\includegraphics{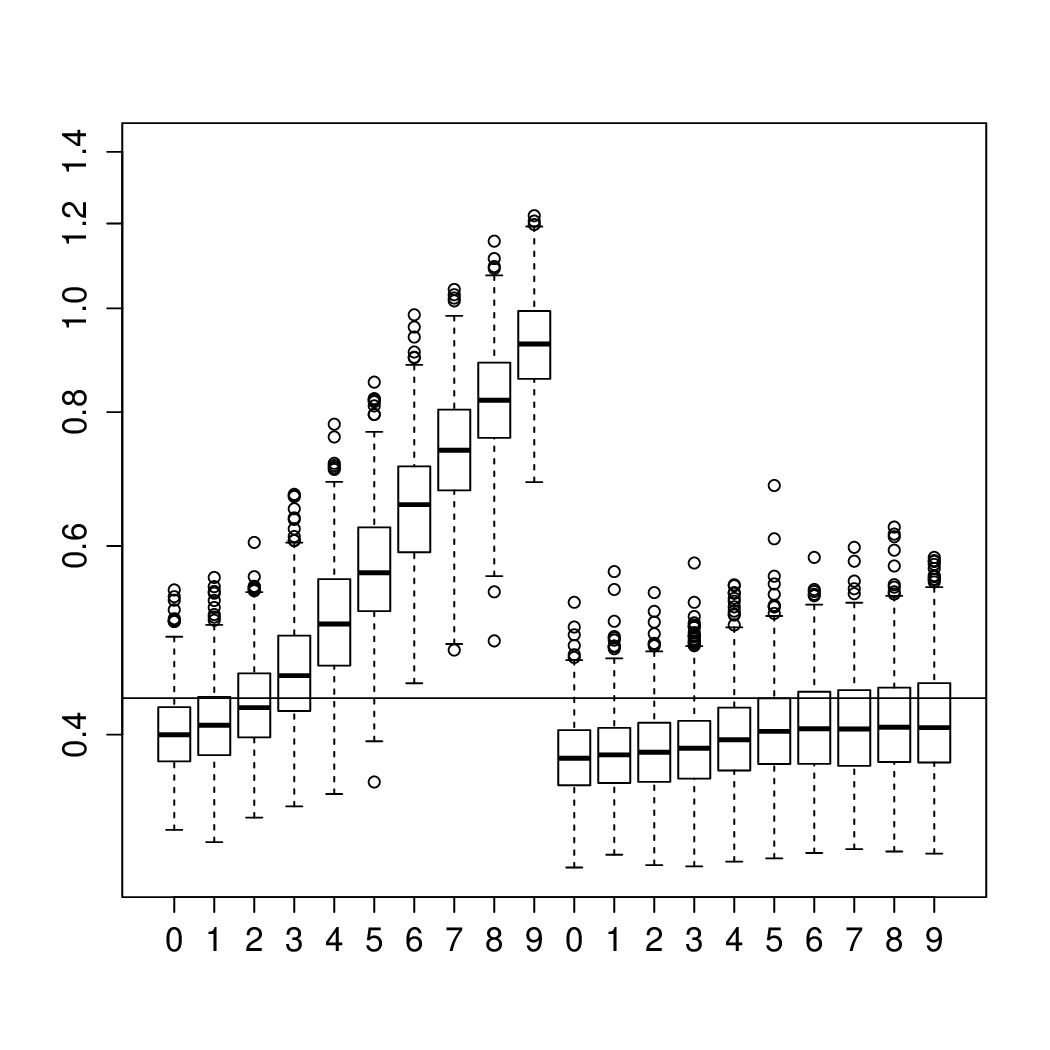}}\\
\resizebox{14.5cm}{5.0cm}{\includegraphics{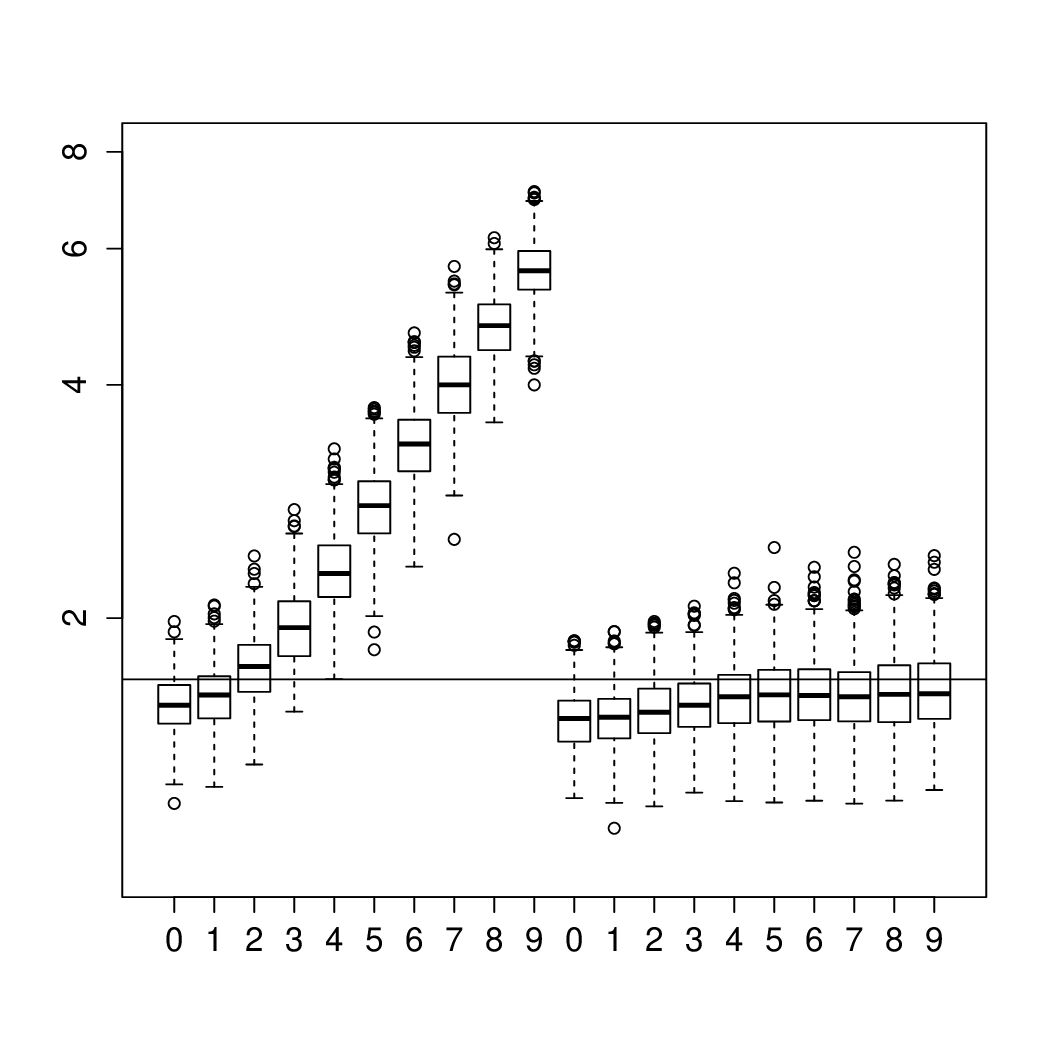}\includegraphics{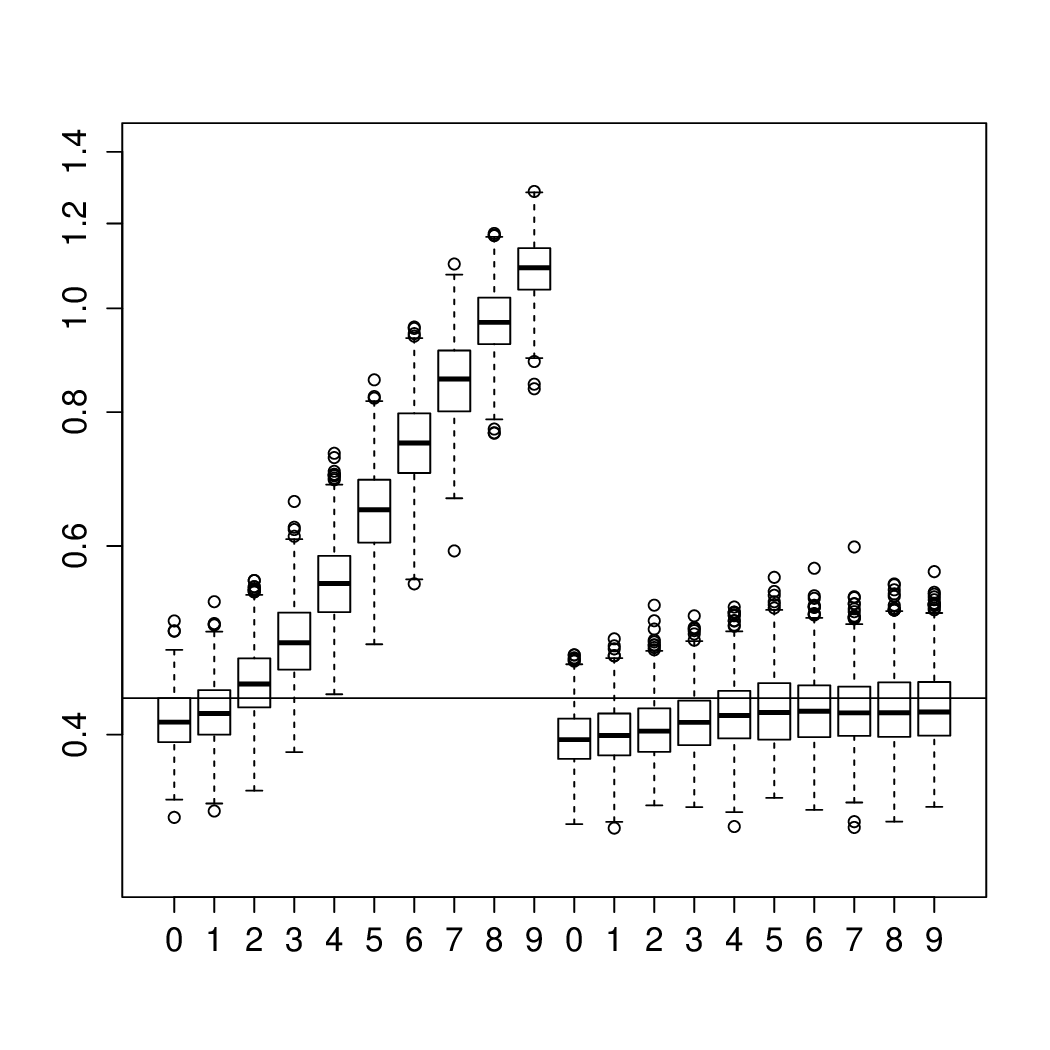}}\\
\resizebox{14.5cm}{5.0cm}{\includegraphics{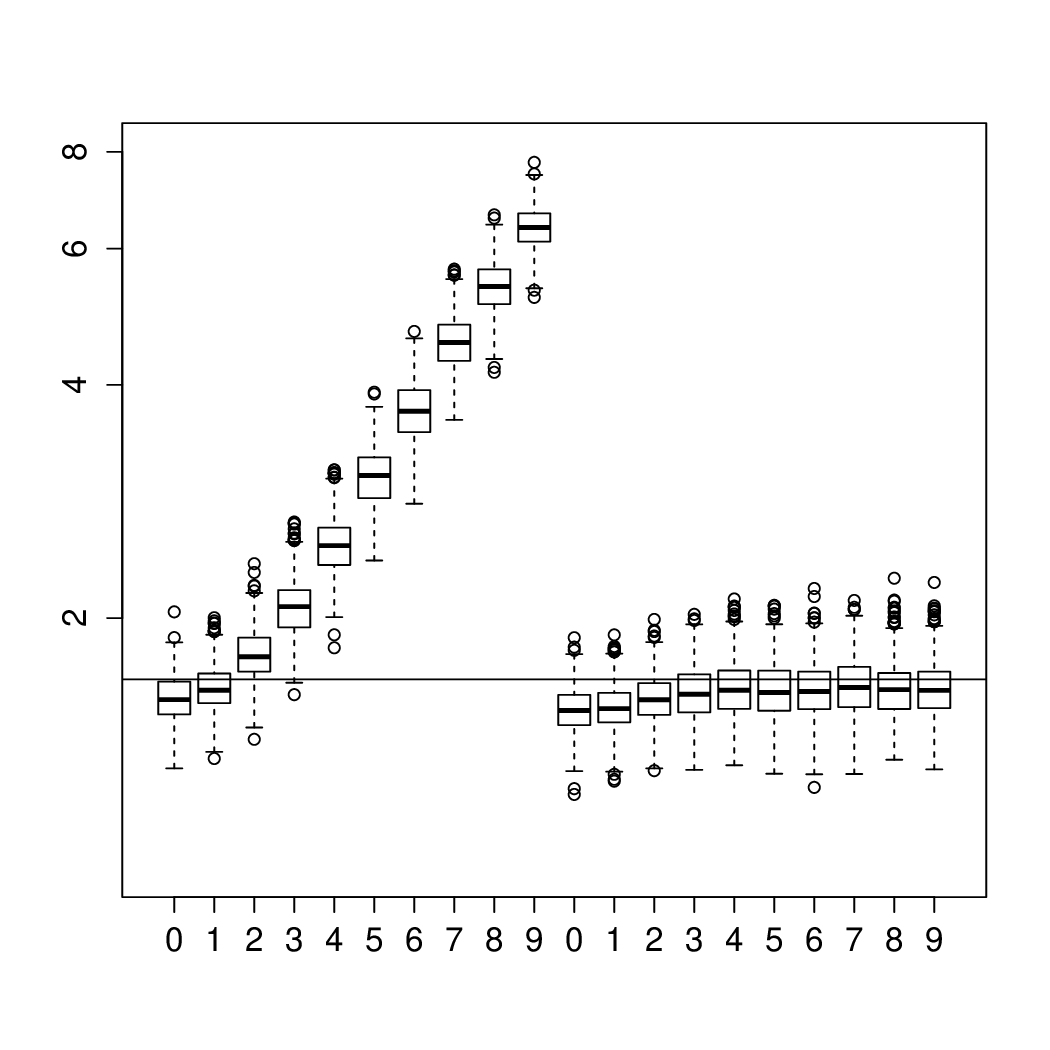}\includegraphics{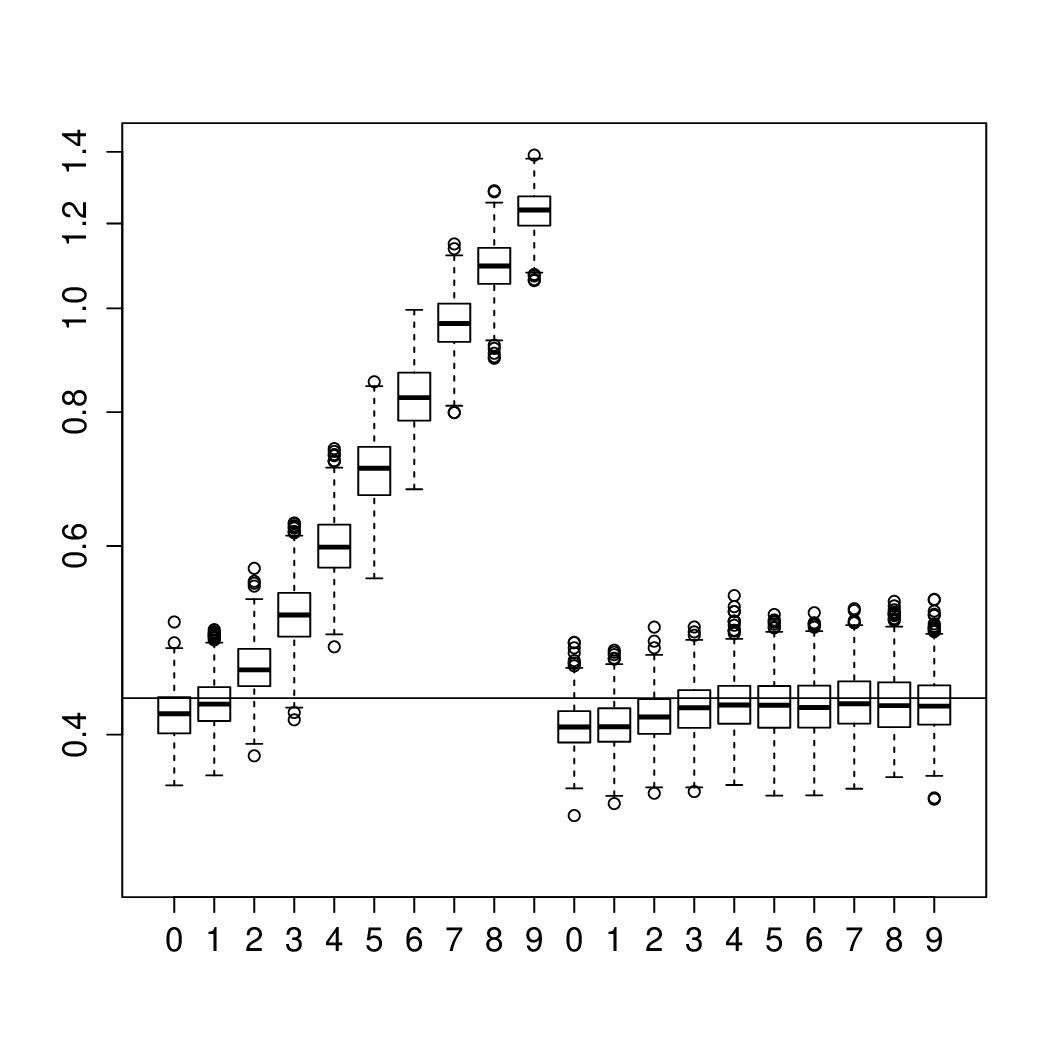}}\\
\caption{\label{fig:asymvarest}Boxplots depicting the simulation results for estimation of the long run standard deviation of the cumulative sum (panels on the left hand side) and the Hodges-Lehmann (right hand side) change-point test statistic in case of a shift of increasing height $j/4$, $j=0,1,\ldots,9$, in the center of a first order autoregressive time series with $\phi_1$=0$\cddot$4 and length
$N=200$ (top), $N=500$ (center) or $N=1000$ (bottom).
Results are on a logarithmic scale as it is natural for scaling factors. Adaptive subsampling as described in the paper applied to the full data set (left hand side of each panel) and median of the subsampling estimates obtained from three non-overlapping subsequences of length $\lfloor N/3\rfloor$ each (right).
 }
\end{figure}

\newpage

Table \ref{tab:anova} provides detailed results for an analysis of variance of the average powers of the different tests obtained under different
autoregressive moving average models and innovation distributions.

\begin{table}[h]{\footnotesize
\caption{\label{tab:anova}
Analysis of variance of the average size-adjusted empirical power of the tests, considering additive effects of the innovation distributions and the autocorrelation model.
Fitted powers are given in percent for autoregressive moving average models with different parameters $(\phi_1,\phi_2,\theta)$ and Gaussian innovations. Additive effects are reported for $t_3$- and $\chi_3^2$-distributed innovations. For instance, the fitted average power of the cumulative sums test C for a shift in the center of $n=100$ $t_3$-distributed white noise observations, with $(\phi_1,\phi_2,\theta)=(0,0,0)$,
is 55$\cddot$5\%+1$\cddot$6\%=57$\cddot$1\%, while for the test $T$ proposed here it is 66$\cdot$1\%.   The adjusted measure of determination is usually larger than 80\% and often larger than 90\%. The standard error of the fit is less than 4\%}{
\begin{tabular}{lrrrrrrrrrrrr}
Test &  C & W & O & T & H && C & W & O & T & H \\
&&&&&&{n=100}\\
$(\phi_1,\phi_2,\theta)$ &\multicolumn{5}{c}{change after $k^\star=50$}&&\multicolumn{5}{c}{change after $k^\star=75$}\\
$(0,0,0)$ & 55$\cddot$5& 55$\cddot$0& 57$\cddot$6& 58$\cddot$4& 54$\cddot$4&& 57$\cddot$2& 52$\cddot$4& 59$\cddot$7& 60$\cddot$8& 56$\cddot$0\\
$(0\cddot 4,0,0)$ & 49$\cddot$7& 50$\cddot$1& 57$\cddot$8& 58$\cddot$1& 49$\cddot$5&& 46$\cddot$5& 31$\cddot$5& 59$\cddot$8& 61$\cddot$5& 44$\cddot$5\\
$(0\cddot 8,0,0)$ & 14$\cddot$2&  8$\cddot$5& 39$\cddot$3& 37$\cddot$4& 10$\cddot$9 && 11$\cddot$8&  6$\cddot$3& 45$\cddot$4& 50$\cddot$9&  6$\cddot$3\\
$(0\cddot 4,0\cddot 3,0)$ & 21$\cddot$3& 12$\cddot$9& 52$\cddot$3& 51$\cddot$6& 16$\cddot$8&&  7$\cddot$4&  4$\cddot$6& 52$\cddot$0& 57$\cddot$7&  4$\cddot$4\\
$(0,0,0\cddot 5)$ & 56$\cddot$3& 57$\cddot$7& 61$\cddot$9& 62$\cddot$1& 56$\cddot$9&& 56$\cddot$6& 45$\cddot$4& 65$\cddot$1& 65$\cddot$5& 56$\cddot$6\\
$(0,0,0\cddot 8)$ & 55$\cddot$3& 55$\cddot$7& 60$\cddot$7& 61$\cddot$2& 55$\cddot$5&& 45$\cddot$7& 34$\cddot$5& 55$\cddot$8& 57$\cddot$5& 45$\cddot$1\\
$(0\cddot 3,0,0\cddot 5)$ & 45$\cddot$6 & 42$\cddot$8& 52$\cddot$2& 53$\cddot$9& 44$\cddot$9 && 34$\cddot$7& 23$\cddot$1& 50$\cddot$3& 53$\cddot$6& 33$\cddot$5\\
$t_3$ & 1$\cddot$6&  6$\cddot$1&  7$\cddot$1&  7$\cddot$7&  6$\cddot$7&&  2$\cddot$3 & 4$\cddot$5&  6$\cddot$7&  7$\cddot$8&  6$\cddot$9\\
$\chi_3^2$ & -1$\cddot$4&  4$\cddot$0& -6$\cddot$2&  4$\cddot$0&  3$\cddot$7&&  -1$\cddot$2&   2$\cddot$8& -10$\cddot$2&   3$\cddot$2&   3$\cddot$8\\
& &&&&&{n=200}\\
&\multicolumn{5}{c}{change after $k^\star=100$}&&\multicolumn{5}{c}{change after $k^\star=150$}\\
$(0,0,0)$&  52$\cddot$7& 53$\cddot$5& 54$\cddot$7& 54$\cddot$6& 52$\cddot$3&&  59$\cddot$1& 54$\cddot$9& 60$\cddot$8& 60$\cddot$8& 58$\cddot$3\\
$(0\cddot 4,0,0)$ &  51$\cddot$7& 53$\cddot$2& 56$\cddot$9& 57$\cddot$2& 52$\cddot$5&&  56$\cddot$1& 49$\cddot$2& 61$\cddot$7& 62$\cddot$2& 56$\cddot$5\\
$(0\cddot 8,0,0)$ &  30$\cddot$7& 24$\cddot$9& 45$\cddot$7& 45$\cddot$1& 28$\cddot$9 &&  21$\cddot$3& 11$\cddot$0& 46$\cddot$6& 49$\cddot$5& 17$\cddot$8\\
$(0\cddot 4,0\cddot 3,0)$ &  45$\cddot$0& 43$\cddot$2& 56$\cddot$8& 55$\cddot$9& 44$\cddot$2&&  33$\cddot$5 &18$\cddot$6& 58$\cddot$4& 59$\cddot$5& 30$\cddot$7\\
$(0,0,0\cddot 5)$ &  57$\cddot$0& 58$\cddot$8& 61$\cddot$1& 61$\cddot$3& 58$\cddot$1&&  50$\cddot$8& 45$\cddot$7& 55$\cddot$9& 56$\cddot$8& 52$\cddot$1\\
$(0,0,0\cddot 8)$ &  56$\cddot$6& 58$\cddot$0& 60$\cddot$5& 60$\cddot$8& 57$\cddot$2&&   52$\cddot$0& 45$\cddot$8& 56$\cddot$9& 57$\cddot$8& 52$\cddot$6\\
$(0\cddot 3,0, 0\cddot 5)$ &  50$\cddot$3& 51$\cddot$0& 55$\cddot$5& 55$\cddot$4& 50$\cddot$2&&  54$\cddot$6& 44$\cddot$0& 60$\cddot$4 &60$\cddot$5& 53$\cddot$9\\
$t_3$ &   2$\cddot$7& 8$\cddot$1& 8$\cddot$6& 8$\cddot$7& 7$\cddot$8&&    2$\cddot$6&  8$\cddot$0&  8$\cddot$0&  8$\cddot$6&  7$\cddot$1\\
$\chi_3^2$ &  -0$\cddot$2&  5$\cddot$2& -1$\cddot$8&  5$\cddot$2 & 3$\cddot$6&&  -0$\cddot$8&  4$\cddot$5& -4$\cddot$8&  4$\cddot$0&  3$\cddot$3\\[2mm]
 \end{tabular}}\label{tab:power}}
\end{table}


\end{document}